\documentclass{amsart}
\usepackage{mathrsfs}
\usepackage{stmaryrd,mathtools}
\usepackage{enumerate}
\usepackage{tikz-cd}
\usepackage[all]{xy}
\usepackage{aliascnt}

\usepackage{shuffle}
\usepackage{ytableau}
\newcommand{\blank}{\phantom{2}}

\usepackage[colorlinks=true, linkcolor=blue, citecolor=blue]{hyperref}
\usepackage{fullpage}
\usepackage{amssymb}

\newtheorem{theorem}{Theorem}[section]

\newaliascnt{lemma}{theorem}
\newtheorem{lemma}[lemma]{Lemma}
\aliascntresetthe{lemma}

\newaliascnt{corollary}{theorem}
\newtheorem{corollary}[corollary]{Corollary}
\aliascntresetthe{corollary}

\newaliascnt{proposition}{theorem}
\newtheorem{proposition}[proposition]{Proposition}
\aliascntresetthe{proposition}

\newaliascnt{potato}{theorem}

\aliascntresetthe{potato}

\newaliascnt{definitionlemma}{theorem}

\aliascntresetthe{definitionlemma}

\newaliascnt{definitionprop}{theorem}
\newtheorem{definitionprop}[definitionprop]{Proposition-Definition}
\aliascntresetthe{definitionprop}

\newaliascnt{conjecture}{theorem}

\aliascntresetthe{conjecture}

\newaliascnt{question}{theorem}

\aliascntresetthe{question}

\theoremstyle{definition}

\newaliascnt{definition}{theorem}
\newtheorem{definition}[definition]{Definition}
\aliascntresetthe{definition}

\newaliascnt{remark}{theorem}
\newtheorem{remark}[remark]{Remark}
\aliascntresetthe{remark}

\newaliascnt{example}{theorem}
\newtheorem{examplex}[example]{Example}
\newenvironment{example}
  {\pushQED{\qed}\examplex}
  {\popQED\endexamplex}
\aliascntresetthe{example}

\newaliascnt{notation}{theorem}
\newtheorem{notation}[notation]{Notation}
\aliascntresetthe{notation}

\definecolor{darkblue}{rgb}{0.6,0,0.1}

\newcommand{\newword}[1]{\textcolor{darkblue}{\textbf{\emph{#1}}}}

\newcommand*{\Scale}[2][4]{\scalebox{#1}{$#2$}}

\usepackage{tikz}
\usetikzlibrary{calc, shapes, backgrounds,arrows,positioning,plotmarks}
\tikzset{>=stealth',
  head/.style = {fill = white, text=black},
  plaque/.style = {draw, rectangle, minimum size = 10mm},
  pil/.style={->,thick},
  junct/.style = {draw,circle,inner sep=0.5pt,outer sep=0pt, fill=black}
  }

\renewcommand{\setminus}{\smallsetminus}

\newcommand{\Z}{\mathbb{Z}}

\newcommand{\C}{\mathbb{C}}
\newcommand{\TT}{\mathbb{T}}

\newcommand{\bP}{\mathbb{P}}
\newcommand{\bA}{\mathbb{A}}

\newcommand{\cO}{\mathcal{O}}

\newcommand{\cC}{\mathcal{C}}

\newcommand{\cI}{\mathcal{I}}

\newcommand{\cA}{\mathcal{A}}

\newcommand{\p}{\mathfrak{p}}

\DeclareMathOperator{\an}{an}

\DeclareMathOperator{\QSym}{QSym}
\DeclareMathOperator{\Sym}{Sym}

\DeclareMathOperator{\Spec}{Spec}
\DeclareMathOperator{\Proj}{Proj}

\DeclareMathOperator{\id}{id}
\DeclareMathOperator{\im}{im}

\DeclareMathOperator{\wt}{wt}

\newcommand{\bx}{\mathbf{x}}
\newcommand{\by}{\mathbf{y}}

\makeatletter
\@namedef{subjclassname@2020}{\textup{2020} Mathematics Subject Classification}
\makeatother

\AtBeginDocument{\def\MR#1{}
}

\begin{document}

\title{James reduced product schemes and double quasisymmetric functions}

\author{Oliver Pechenik}
\thanks{OP was partially supported by a Discovery Grant (RGPIN-2021-02391) and Launch Supplement (DGECR-2021-00010) from the Natural Sciences and Engineering Research Council of Canada.}
\address[OP]{Department of Combinatorics \& Optimization, University of Waterloo, Waterloo ON N2L3G1, Canada}
\email{oliver.pechenik@uwaterloo.ca}

\author{Matthew~Satriano}
\thanks{MS was partially supported by a Discovery Grant from the
  National Science and Engineering Research Council of Canada and a Mathematics Faculty Research Chair from the University of Waterloo.}
\address[MS]{Department of Pure Mathematics, University
  of Waterloo, Waterloo ON N2L3G1, Canada}
\email{msatrian@uwaterloo.ca}

\date{\today}

\begin{abstract}
Symmetric function theory is a key ingredient in the Schubert calculus of Grassmannians. Quasisymmetric functions are analogues that are similarly central to algebraic combinatorics, but for which the associated geometry is poorly developed. Baker and Richter (2008) showed that $\QSym$ manifests topologically as the cohomology ring of the loop suspension of infinite projective space or equivalently of its combinatorial homotopy model, the James reduced product $J\C\bP^\infty$. In recent work, we used this viewpoint to develop topologically-motivated bases of $\QSym$ and initiate a Schubert calculus for $J\C\bP^\infty$ in both cohomology and $K$-theory.

Here, we study the torus-equivariant cohomology of $J\C\bP^\infty$. We identify a cellular basis and introduce double monomial quasisymmetric functions as combinatorial representatives, analogous to the factorial Schur functions and double Schubert polynomials of classical Schubert calculus. We also provide a combinatorial Littlewood--Richardson-type rule for the structure coefficients of this basis.

Furthermore, we introduce an algebro-geometric analogue of the James reduced product construction. In particular, we prove that the James reduced product of a complex (quasi-)projective variety also carries the structure of a (quasi-)projective variety.
\end{abstract}

\maketitle

\numberwithin{theorem}{section}
\numberwithin{lemma}{section}
\numberwithin{corollary}{section}
\numberwithin{proposition}{section}
\numberwithin{conjecture}{section}
\numberwithin{question}{section}
\numberwithin{remark}{section}
\numberwithin{definition}{section}
\numberwithin{example}{section}
\numberwithin{notation}{section}
\numberwithin{equation}{section}

\section{Introduction}
\label{sec:intro}
For $X$ a pointed topological space satisfying mild hypotheses, the \newword{James reduced product} $JX$ (due to I.~James \cite{James}) is a combinatorial homotopy model for the loop suspension $\Omega \Sigma X$ and plays an important role in homotopy theory. The points of $JX$ are the elements of the free monoid on $X$ with the basepoint $e$ as neutral element. That is,
\[
JX = \Bigg( \coprod_{i \geq 0} X^i \Bigg) / {\sim},
\]
where
\[
(x_1, \dots x_k) \sim (y_1, \dots, y_\ell)
\]
if and only if $(x_1, \dots x_k)$ and $(y_1, \dots, y_\ell)$ differ by inserting and/or removing instances of the basepoint $e$. The James reduced product $JX$ has a filtration
\[
J_0X \subset J_1X \subset \dots \subset JX,
\]
where $J_kX$ consists of those points of $JX$ that can be represented by a list of length $k$. For a textbook treatment of James reduced products, see \cite[VII.2]{Whitehead}.

Here, we introduce an algebro-geometric analogue of the James reduced product construction, which allows us to apply powerful tools of algebraic geometry. Our main geometric result is the following.

\begin{theorem}\label{thm:JnPm-intro}
If $X$ is a complex quasi-projective (resp.\ projective, affine) variety, then $J_nX$ also carries the structure of a quasi-projective (resp.\ projective, affine) variety, while $JX$ carries the structure of an ind-variety.
\end{theorem}

We give a more precise (and more general) version of Theorem \ref{thm:JnPm-intro} in Section \ref{sec:variety}. Specifically, in Proposition \ref{prop:OJnX-sheaf}, we give a James reduced product construction $J_nX$ associated to any pair $(X,\cI)$ where $X$ is a scheme satisfying mild hypotheses and $\cI$ is an ideal sheaf. We show $J_nX$ is functorial (see Proposition \ref{prop:functorialityJnX}) and preserves affineness, projectivity, quasi-affineness, and quasi-projectivity (see Corollary \ref{cor:JnX-affine-scheme}, Proposition \ref{prop:JnX-proj-scheme}, and Proposition \ref{prop:JnX-qproj-scheme}). Additionally, we explicitly describe the coordinate ring (respectively, homogeneous coordinate ring) for $J_nX$ when $X$ is affine (respectively, projective). Finally, when $X$ is a complex quasi-projective variety and $\cI$ is the ideal sheaf of a base point $e\in X$, we prove in Theorem \ref{thm:JnXan} that the analytic topology on $J_nX$ agrees with the topological James reduced product. Thus, our construction endows the topological $J_nX$ with algebraic structure.

While we expect Theorem~\ref{thm:JnPm-intro} to be useful more generally, our main application is to the torus-equivariant cohomology of $J(\C \bP^\infty)$. We obtain control of this cohomology ring by exhibiting an explicit paving of $J(\C \bP^\infty)$ by torus-stable affines. It is worth noting that $J(\C \bP^\infty)$ has infinitely-many torus-stable curves and standard tools introduced by Goresky, Kottwitz, and MacPherson \cite{GKM} and Chang--Skelbred \cite{ChangSkjelbred} do not directly apply, see Remark \ref{rmk:ChangSkjelbred} for further details.

Our interest in $J(\C \bP^\infty)$ is rooted in the observation of A.~Baker--B.~Richter \cite{Baker.Richter} that the ordinary cohomology of $J(\C \bP^\infty)$ recovers the algebra $\QSym$ of \emph{quasisymmetric functions}. Quasisymmetric functions are an extension of classical symmetric function theory introduced by R.~Stanley \cite{Stanley} in service of enumerative combinatorics. Greatly developed by I.~Gessel \cite{Gessel}, they now occupy a central place in algebraic combinatorics more generally, with connections, for example, to representation theory of $0$-Hecke algebras \cite{Choi.Kim.Nam.Oh, Duchamp.Krob.Leclerc.Thibon, Searles:extended, Tewari.vanWilligenburg}, Macdonald theory \cite{Corteel.Haglund.Mandelshtam.Mason.Williams}, time series \cite{Diehl.EbrahimiFard.Tapia}, and graph polynomials \cite{Shareshian.Wachs}. See \cite{Mason} for a survey of recent developments in quasisymmetric function theory and \cite{Billey.McNamara} for a more historical perspective.

For combinatorial purposes, it is insufficient to have abstract isomorphisms of algebras, such as that provided between $H^\star(J(\C \bP^\infty))$ and $\QSym$ in \cite{Baker.Richter}. It is necessary further to track the images of distinguished bases. In previous work \cite{Pechenik.SatrianoKthy}, we showed that the Baker--Richter isomorphism identifies the canonical cellular basis of $H^\star(J(\C \bP^\infty))$ with the classical \emph{monomial quasisymmetric function} basis of $\QSym$, yielding an analogue of Schubert calculus for $J(\C \bP^\infty)$ and related spaces.

Traditionally, Schubert calculus studies the cohomology rings of Grassmannians and other flag manifolds through the construction of combinatorial models, such as symmetric function theory and its asymmetric analogues (see \cite{Pechenik.Searles:survey} for a survey of this perspective). Modern Schubert calculus is especially interested in developing analogous combinatorial theories for richer cohomology theories, especially $K$-theory (e.g., \cite{Buch}) and torus-equivariant theories (e.g., \cite{Knutson.Tao:HT,Pechenik.Yong}). Here, for example, \emph{Schur functions} (for Grassmannians) and \emph{Schubert polynomials} (for complete flag manifolds) are replaced in $K$-theory with \emph{Grothendieck polynomials} and in torus-equivariant cohomology with \emph{factorial Schur functions} and \emph{double Schubert polynomials}, respectively.

As initiated in \cite{Lam.Pylyavskyy}, it has long been desired to obtain $K$-theoretic and torus-equivariant analogues in quasisymmetric function theory in analogy with well-studied such analogues within symmetric and asymmetric function theory. In \cite{Pechenik.SatrianoKthy}, we developed a $K$-theoretic analogue of the Baker--Richter isomorphism, together with a ``cellular'' basis of the $K$-theory ring of $J(\C \bP^\infty)$, to obtain the first topologically-motivated $K$-analogues in quasisymmetric function theory.

Here, we pursue the other most-desired analogue, obtaining the first torus-equivariant versions of quasisymmetric functions. Specifically, we introduce double monomial quasisymmetric functions $M_\alpha(\bx,\by)$, which we show correspond to a cellular basis of $H^\star_T(J(\C \bP^\infty))$ that we develop. Letting $\Lambda=H^*_T(pt) = \mathbb{Z}[y_1, y_2, \dots]$, we define, for $\alpha=(a_1,\dots,a_k)$ a sequence of positive integers, the \newword{double monomial quasisymmetric function} as

\begin{equation}\label{eq:double_monomial}
M_\alpha(\bx,\by)=\sum_{1\leq i_1<\dots<i_k}\prod_{\ell=1}^k\prod_{1\leq j_\ell\leq a_\ell}(x_{i_\ell}-y_{j_\ell})\ \in\ \Lambda[[x_1-y_1,x_2-y_1,\dots]]_{\textrm{gr}}.\footnote{Note we are only taking power series in $x_i-y_1$, not $x_i-y_j$. See Example \ref{rmk:doubleM-are-in-appropriate-power-series} for further details.}
\end{equation}
 Double monomial quasisymmetric functions recover ordinary monomial quasisymmetric functions by setting the $y$-variables to $0$, reflecting forgetting the $T$-action.

Our second main theorem is the following.

\begin{theorem}
\label{james-reduced-cells->monomial-basis}
We have a graded $\Lambda$-module isomorphism
\[
H^*_T(J(\C\bP^\infty); \Z)\simeq\QSym(\Lambda)
\]
identifying the cellular basis element $x_\alpha$ with the double monomial quasisymmetric function $M_\alpha(\bx,\by)$.
\end{theorem}

Lastly, we further the analogy with Schubert calculus in Theorem \ref{thm:product-formula}, giving a cancellation-free combinatorial rule for multiplying double monomial quasisymmetric functions $M_\alpha(\bx,\by)$ (analogous to the classical Littlewood--Richardson rule for Schur functions). The rule is defined in terms of an \emph{overlapping shuffle product} (see \cite{Hazewinkel}) and relies on work in equivariant Schubert calculus by H.~Thomas--A.~Yong \cite{Thomas.Yong18}.

\medskip
{\bf This paper is organized as follows.} In Section~\ref{sec:variety}, we prove a more precise version of Theorem~\ref{thm:JnPm-intro}, endowing the James reduced product with the structure of an algebraic variety. In Section~\ref{sec:cohomology}, we first give a general result (Proposition~\ref{prop:base-change-HT}) about torus-equivariant cohomology of projective varieties with affine pavings. Then we narrow our focus to $X = \C\bP^\infty$ and establish our cellular basis result, Theorem~\ref{james-reduced-cells->monomial-basis}. Finally, Section~\ref{sec:LR} gives a Littlewood--Richardson-type rule for the equivariant cohomology of $J(\C\bP^\infty)$, or equivalently for the multiplication of double monomial quasisymmetric functions, together with some examples of its use.

\section*{Acknowledgments}
This project was inspired by a question of Frank Sottile and benefitted greatly from conversations with Dori Bejleri, Dan Edidin, Changho Han, Nathan Ilten, Christian Liedtke, David McKinnon, Doug Park, and Ben Webster. We are indebted to Dave Anderson, Darij Grinberg, Megumi Harada, and, especially, Matthias Franz many enlightening discussions on equivariant cohomology. We wholeheartedly thank Dave Anderson for detailed comments on an earlier version of our paper which shortened some of our arguments. Lastly, we thank the anonymous referees for their many helpful suggestions, especially their suggestion that we generalize our James reduced product construction to the case of schemes. This made several of our results both cleaner and more general.

\section{The James reduced product as a scheme
}\label{sec:variety}

In order to apply standard tools in equivariant cohomology, we prove that the James reduced product carries the structure of an algebraic variety. In Subsection \ref{subsec:qsymdef}, we introduce a generalization of the ring of quasisymmetric functions; this will serve as the local models for $J_nX$. In Subsection \ref{subsec:JnXdef}, we define the James reduced product scheme $J_nX$, show that our construction is functorial, and prove that it preserves (quasi-)affineness and (quasi-)projectivity. Finally, in Subsection \ref{subsec:JnXan}, we show how our construction recovers the topological James reduced product.

\subsection{Generalized ring of quasisymmetric functions}\label{subsec:qsymdef}

Throughout Section \ref{sec:variety}, let $R$ be a ring. We use the following notation:~if $A$ is an $R$-algebra and $f\in A$, then
\begin{equation}\label{eqn:f-super-k}
f^{(k)} \coloneqq 1\otimes\dots\otimes1\otimes f\otimes1\dots\otimes 1\, \in\, A\otimes_R\dots\otimes_R A \eqqcolon A^{\otimes n}
\end{equation}
with $f$ in the $k$-th tensor factor.

We introduce the following algebra.

\begin{definition}\label{def:qsym-ring-at-m}
Let $A$ be an $R$-algebra and $I\subset A$ an ideal. Then the $R$-algebra of \newword{$n$-truncated quasisymmetric functions with respect to $I$} is defined as the $R$-subalgebra
\[
\QSym_{n,I}(A/R)\subset A^{\otimes n}
\]
generated by all expressions of the form
\[
\sum_{1\leq k_1<\dots<k_s\leq n}f_1^{(k_1)}\dots f_s^{(k_s)}
\]
where $0\leq s\leq n$ and $f_1,\dots,f_s\in I$.

Furthermore, if $I$ is the kernel of an $R$-algebra map $e\colon A\to R$, then $A^{\otimes n}\xrightarrow{\id\otimes e} A^{\otimes(n-1)}$ maps $\QSym_{n,I}(A/R)$ to $\QSym_{n-1,I}(A/R)$. In this case, we define the $R$-algebra of \newword{quasisymmetric functions with respect to $I$} to be
\[
\QSym_I(A/R) \coloneqq \lim_n \QSym_{n,I}(A/R).
\]
\end{definition}

\begin{example}\label{qsym-ring-at-m:A1}
Our definition above recovers the usual quasisymmetric functions (see \cite{Mason, Pechenik.Searles:survey}) in $n$ variables with coefficients in $R$ when $A=R[x]$ and $I=(x)$:
\[
\QSym_n(R)=\QSym_{n,(x)}(R[x]/R).\qedhere
\]
\end{example}

Given an $R$-scheme $X$, we will construct its James reduced product as the relative spectrum of a sheaf of algebras on the symmetric product $\Sym^n(X/R)$. Recall that $\Sym^n(X/R)$ is defined as the quotient of $X^n \coloneqq X\times_R\dots\times_R X$ by the action of the symmetric group $S_n$; in general it is only an algebraic space.

\begin{lemma}\label{l:SymnX-scheme}
Let $n\geq1$ and let $X$ be an $R$-scheme such that for all $x_1,\dots,x_n\in X$, there exists an open affine $U\subset X$ containing all $x_i$. Then $\Sym^n(X/R)$ is a scheme. Furthermore, there is an open affine cover $X=\bigcup_iU_i$ such that $\Sym^n(X/R)$ is covered by the open affine subsets $\Sym^n(U_i/R)$.
\end{lemma}
\begin{proof}
Given any point $(x_1,\dots,x_n)\in X^n$, by hypothesis, there exists an open affine $U\subset X$ such that $(x_1,\dots,x_n)\in U^n$. Since $U^n\subset X^n$ is open affine and $S_n$-invariant, we see $\Sym^n(U/R)\subset\Sym^n(X/R)$ is an open affine subset. This proves the result.
\end{proof}

\begin{lemma}\label{l:SymnX-schemeDhcover}
Let $A$ be an $R$-algebra, $X=\Spec A$, and $U\subset X$ be open. Then there is an open affine cover $U=\bigcup_i U_i$ such that $\Sym^n(U/R)=\bigcup_i\Sym^n(U_i/R)$ and each $U_i$ is of the form $D(h_i)=\Spec(A_{h_i})$ for $h_i\in A$.
\end{lemma}
\begin{proof}
We show that for any collection of points $x_1,\dots,x_n\in U$, there exists some $h\in A$ such that $x_1,\dots,x_n\in D(h)\subset U$. The points correspond to prime ideals $\p_1,\dots,\p_n\subset A$ and $X\setminus U$ corresponds to an ideal $J\subset A$ with $J\not\subset\p_i$ for any $i$. Then we may choose $h\in J$ such that $h\notin\p_1\cup\dots\cup\p_n$, hence $D(h)$ has the desired property.
\end{proof}

\begin{remark}\label{rmk:FA-scheme}
A scheme $X$ is said to be an \newword{FA-scheme} if every finite collection of points is contained in an open affine, i.e., FA-schemes satisfy the hypothesis of Lemma \ref{l:SymnX-scheme} for all $n$. FA-schemes include all quasi-projective varieties, and more generally, any quasi-projective scheme over an affine scheme. See \cite[Section 2.2]{GabberLiuLorenzini} for more examples.
\end{remark}

The following example shows that $\Sym^n(X/R)$ can fail to be a scheme if the hypotheses of Lemma \ref{l:SymnX-scheme} are not met.

\begin{example}
Let $X$ be the non-separated affine line over a field $k$. Then $\Sym^2(X/k)$ is not a scheme. To see this, let the two origins in $X$ be denoted by $0_1$ and $0_2$. We claim there is no affine open $S_2$-invariant subset $W\subset X^2$ containing $(0_1,0_2)$. To see this, intersect $W$ with the four planes covering $X^2$ to obtain open subsets $W_{i,j}\subset\bA^2$ such that $W_{i,j}\setminus V(x)=W_{3-i,j}\setminus V(x)$ and $W_{i,j}\setminus V(y)=W_{i,3-j}\setminus V(y)$. So, $W_{1,2}\setminus V(xy)=W_{2,1}\setminus V(xy)$. Since both $W_{1,2}$ and $W_{2,1}$ must contain the origin (since $W$ contains $(0_1,0_2)$ and $(0_2,0_1)$), we see that $W_{1,1}$ and $W_{1,2}$ agree away from finitely many points. Thus, $W$ is not separated, hence not affine.\qedhere
\end{example}

We end this subsection by showing that the formation of $\QSym_{n,I}(A/R)$ commutes with affine open immersions on $\Spec A$. This will be used to glue together our local models of $J_nX$. We first require a technical lemma.

\begin{lemma}\label{l:technical-expressions-in-QSym}
Let $A$ be an $R$-algebra and let $I\subset A$ be an ideal. Let $s>0$ and $s+t\leq n$. If $f_1,\dots,f_s\in I$ and $h\in A$, then
\begin{equation}\label{eqn:technical-expressions-in-QSym}
\sum_{1\leq k_1<\dots<k_s\leq n}\bigg(\prod_{i=1}^sf_i^{(k_i)}\bigg)\bigg(\prod_{\stackrel{\{k_i\}\cap\{\ell_j\}=\varnothing}{1\leq\ell_1<\dots<\ell_t\leq n}}h^{(\ell_j)}\bigg)\in\QSym_{n,I}(A/R).
\end{equation}
\end{lemma}
\begin{proof}
We prove the result by induction on $t$. Note that if $t=0$, then the expression \eqref{eqn:technical-expressions-in-QSym} is, by definition, in $\QSym_{n,I}(A/R)$.

For $t>0$, note that
\[
\bigg(\sum_{1\leq k_1<\dots<k_s\leq n}\prod_{i=1}^sf_i^{(k_i)}\bigg)\bigg(\sum_{1\leq\ell_1<\dots<\ell_t\leq n}\prod_{j=1}^t h^{(\ell_j)}\bigg)\in\QSym_{n,I}(A/R)
\]
and that its difference with \eqref{eqn:technical-expressions-in-QSym} is given by
\begin{equation}\label{eqn:partialproduct}
\sum_{N\geq1}(-1)^N\sum_{1\leq k_1<\dots<k_s\leq n}\bigg(\prod_{i=1}^sf_i^{(k_i)}\bigg)\bigg(\prod_{\stackrel{|\{k_i\}\cap\{\ell_j\}|=N}{1\leq\ell_1<\dots<\ell_t\leq n}}\prod_{j=1}^t h^{(\ell_j)}\bigg).
\end{equation}
Taking into account the indices on $k$ in the intersection $\{k_i\}\cap\{\ell_j\}$, for fixed $N$, the inner sum in \eqref{eqn:partialproduct} can be rewritten as
\begin{equation}\label{eqn:smaller-t-expression}
\sum_{1\leq m_1<\dots<m_N\leq s\ }\sum_{1\leq k_1<\dots<k_s\leq n}\bigg(\prod_{i=1}^s(f_i\delta_i)^{(k_i)}\bigg)\bigg(\prod_{\stackrel{\{k_i\}\cap\{\ell_j\}=\varnothing}{1\leq\ell_1<\dots<\ell_{t-N}\leq n}}h^{(\ell_j)}\bigg)
\end{equation}
where $\delta_{m_j}=h$, and $\delta_i=1$ for $i\notin\{m_j\}$. Since $f_i\in I$, we see $f_i\delta_i\in I$. Since $t-N<t$, by induction expression \eqref{eqn:smaller-t-expression} is in $\QSym_{n,I}(A)$.
\end{proof}

\begin{lemma}\label{l:QSymOpenImmersions}
Let $\Spec B\to\Spec A$ be an open immersion of $R$-schemes and $I\subset A$ an ideal. Then the natural map
\[
\Sym^n(B/R)\otimes_{\Sym^n(A/R)}\QSym_{n,I}(A/R)\xrightarrow{\simeq}\QSym_{n,IB}(B/R)
\]
is an isomorphism.
\end{lemma}
\begin{proof}
Let $\varphi\colon A\to B$ be the induced map on rings. We will suppress $R$ in our notation for $\Sym$ and $\QSym$. To prove the map in the statement of the lemma is an isomorphism, it suffices to look locally on $\Spec\Sym^n(B)$. By Lemma \ref{l:SymnX-schemeDhcover}, $\Spec\Sym^n(B)$ is covered by opens of the form $\Spec\Sym^n(A_h)$ where $h\in A$. Note that $A_h=B_{\varphi(h)}$. Thus, if we can prove the lemma for open immersions of the form $\Spec(C_g)\to\Spec C$, then applying it to the case where $(C,g)$ is $(A,h)$ and $(B,\varphi(h))$, we will be done.

We have therefore reduced to the case where $B=A_h$. We begin by showing injectivity of the map in the statement of the lemma. Consider the following commutative diagram
\[
\xymatrix{
A^{\otimes n}\ar[r] & \Sym^n(B)\otimes_{\Sym^n(A)} A^{\otimes n}\ar[r]^-{\simeq}_-{\beta} & B^{\otimes n}\\
\QSym_{n,I}(A)\ar@{^{(}->}[u]\ar[r] & \Sym^n(B)\otimes_{\Sym^n(A)}\QSym_{n,I}(A)\ar[u]^-{\alpha}\ar[r]^-{\gamma} & \QSym_{n,IB}(B)\ar@{^{(}->}[u]\\
\Sym^n(A)\ar[r]\ar@{^{(}->}[u] & \Sym^n(B)\ar[u]\ar@{^{(}->}[ur]
}
\]
Since $\Spec\Sym^n(B)\to\Spec\Sym^n(A)$ is an open immersion it is flat. Hence, $\alpha$ is injective as it is the base change of the inclusion $\QSym_{n,I}(A)\subset A^{\otimes n}$. Since $\beta$ is an isomorphism, we see $\gamma$ is injective as desired.

It remains to prove surjectivity of $\gamma$. This can be checked on generators of $\QSym_{nI}(A_h)$, all of which are of the form
\begin{equation}\label{eqn:QSymAh-generator}
\sum_{1\leq k_1<\dots<k_s\leq n}\prod_{i=1}^s(f_i/h^{m_i})^{(k_i)}
\end{equation}
with $f_i\in I$ and $m_i\geq0$. Letting $m\geq \max_i\{m_i\}$, we see $H \coloneqq \prod_{i=1}^n(h^{-m})^{(i)}\in\Sym^n(A_h)$ and
\[
F \coloneqq \sum_{1\leq k_1<\dots<k_s\leq n}\bigg(\prod_{i=1}^s(f_ih^{m-m_i})^{(k_i)}\bigg)\bigg(\prod_{1\leq\ell_1<\dots<\ell_{n-s}\leq n}(h^m)^{(\ell_j)}\bigg)\in\QSym_{n,I}(A)
\]
by Lemma \ref{l:technical-expressions-in-QSym}. We see that the expression \eqref{eqn:QSymAh-generator} is given by $\gamma(H\otimes F)$; hence surjectivity of $\gamma$ follows.
\end{proof}

\subsection{The James reduced product scheme}
\label{subsec:JnXdef}

Throughout the rest of Section \ref{sec:variety}, we work over a fixed base ring $R$ and suppress $R$ in our notation for $\Sym$ and $\QSym$. Our goal in this section, is to construct the scheme $J_nX$. Since we are interested in cases where $J_nX$ is a scheme for all $n$, we restrict attention to FA-schemes (see Remark~\ref{rmk:FA-scheme}).

To globalize Definition \ref{def:qsym-ring-at-m}, we consider an FA-scheme $X$ over $R$, equipped with an ideal sheaf $\cI$. The case of topological James reduced products will correspond to letting $X$ be a complex variety and $\cI$ the ideal sheaf of the base point. Note that
if $X=\Spec A$, then the quotient map $\pi_{n,X}\colon X^n\to\Sym^nX$ factors as $X^n\xrightarrow{q_{n,X}}\Spec\QSym_{n,\cI(X)}(A)\xrightarrow{p_{n,X}} \Sym^nX$. We frequently suppress the subscript $X$ in these maps.

\begin{definitionprop}\label{prop:OJnX-sheaf}
Let $\cI$ be an ideal sheaf on an FA-scheme $X$ over $R$, and let $\pi_n\colon X^n\to\Sym^nX$ be the quotient map. Then there is a unique coherent sheaf of $\cO_{\Sym^n(X)}$-algebras $\cA$ such that
\begin{enumerate}
\item\label{thm:OJnX--subalg} $\cA\subset(\pi_n)_*\cO_{X^n}$ is a subsheaf of $\cO_{\Sym^nX}$-subalgebras, and
\item\label{thm:OJnX--affine} for every affine open $U\subset X$, we have
\[
\cA|_{\Sym^nU}=(p_{n,U})_*\cO_{\Spec \QSym_{n,\cI(U)}(\cO_X(U))}
\]
as $\cO_{\Sym^nU}$-subalgebras of $(\pi_{n,U})_*\cO_{U^n}$.
\end{enumerate}
We may therefore define the scheme
\[
J_n(X/R,\cI)\coloneqq \Spec_{\Sym^nX}\cA;
\]
when $\cI$ and $R$ are understood from context, we simply denote this as $J_nX$.
Furthermore, $\pi_n$ factors as a composition of finite morphisms
\[
X^n\xrightarrow{q_{n,X}} J_nX\xrightarrow{p_{n,X}} \Sym^nX.
\]
with $\cA=(p_{n,X})_*\cO_{J_nX}$.
\end{definitionprop}
\begin{proof}
Lemma \ref{l:SymnX-scheme} shows that $\Sym^n(X)$ is covered by open affines of the form $\Sym^n(U)$ with $U\subset X$ open affine. Since $X$ is separated (see, e.g., \cite[\S 2.2]{GabberLiuLorenzini}), $W \coloneqq U\cap V$ is affine whenever $U$ and $V$ are open affines in $X$; furthermore, $\Sym^n(U)\cap\Sym^n(V)=\Sym^n(W)$. Hence, Lemma \ref{l:QSymOpenImmersions} shows that the restrictions of $(p_{n,U})_*\cO_{\Spec \QSym_{n,\cI(U)}(\cO_X(U))}$ and $(p_{n,V})_*\cO_{\Spec \QSym_{n,\cI(V)}(\cO_X(V))}$ to $\Sym^n(W)$ are both equal $(p_{n,W})_*\cO_{\Spec \QSym_{n,\cI(W)}(\cO_X(W))}$ as subsheaves of $(\pi_{n,W})_*\cO_{W^n}$. Thus, the coherent sheaves of $\cO_{\Sym^nU}$-subalgebras from part (\ref{thm:OJnX--affine}) glue to form a unique coherent sheaf of $\cO_{\Sym^nX}$-subalgebra as in (\ref{thm:OJnX--subalg}).

Defining $J_nX$ as in the statement of the result and letting $p_{n,X}\colon J_nX\to \Sym^nX$ be the structure map, coherence of $\cA$ implies $p_{n,X}$ is finite. Since we have inclusions of $\cO_{\Sym^nX}$-algebras
\[
\cO_{\Sym^nX}\subset \cA\subset (\pi_n)_*\cO_{X^n},
\]
we see $\pi_n$ factors as stated, where finiteness of $q_{n,X}$ follows from finiteness of $\pi_n$ and $p_{n,X}$.
\end{proof}

By construction we have the following property, which we highlight for reference.

\begin{corollary}[{James reduced product of affine schemes}]
\label{cor:JnX-affine-scheme}
Let $X$ be an $R$-scheme equipped with an ideal sheaf $\cI$. If $X$ is affine, then $J_nX$ is as well. Specifically, if $A$ is a ring and $I\subset A$ is an ideal, then
\[
J_n(\Spec A)=\Spec\QSym_{n,I}(A).
\]
\end{corollary}

We next show that projectivity is preserved by our James reduced product construction.

\begin{proposition}[{James reduced product of projective schemes}]
\label{prop:JnX-proj-scheme}
Let $\cI$ be an ideal sheaf on a scheme $X$. If $X$ is projective over $R$, then $J_nX$ is as well. Specifically, if $I$ is a graded ideal in a graded $R$-algebra $S$, then
\[
J_n(\Proj_R(S))=\Proj_R(\QSym_{n,I}(S));
\]
here the grading on $\QSym_{n,I}(S)$ comes from its inclusion into the graded ring $S^{\otimes n}$.
\end{proposition}
\begin{proof}
In this proof, we follow standard notation as in \cite[\S2.7]{HartshorneAG}. The inclusions of graded rings $\Sym^n(S)\subset\QSym_{n,I}(S)\subset S^{\otimes n}$ yield maps
\[
X^n\to\Proj_R(\QSym_{n,I}(S))\to\Sym^n(X).
\]
By a variant of Lemma \ref{l:SymnX-schemeDhcover}, we may cover $\Sym^n(X)$ by open affines of the form $\Sym^n(D_+(h))$ where $h\in S_+$ is homogeneous. By construction,
\[
\Proj_R(\QSym_{n,I}(S))\times_{\Sym^n(X)}\Sym^n(D_+(h))=\Spec\QSym_{n,I}(S)_{(h^{(1)}\dots h^{(n)})}=\Spec\QSym_{n,I_{(h)}}(S_{(h)}),
\]
which is precisely $J_n(\Proj_R(S))\times_{\Sym^n(X)}\Sym^n(D_+(h))$. Since this holds for all homogeneous $h\in S_+$, the above local identifications glue to show that $J_n(\Proj_R(S))=\Proj_R(\QSym_{n,I}(S))$.
\end{proof}

The following is a simple, yet illustrative, example.

\begin{example}\label{ex:J2A1}
Working over a field $k$ and equipping $\bA^1$ with the ideal sheaf corresponding to the base point $0\in\bA^1$, we see $J_2\bA^1=\Spec \QSym_{2,(x)}k[x]$, where $\QSym_{2,(x)}k[x]$ is the usual ring of quasisymmetric functions in two variables. One checks that $\QSym_{2,(x)}k[x]=k[x_1x_2,x_1+x_2,x_1^2x_2]\subset k[x_1,x_2]$. There is one relation among these three generators, so letting $t=x_1x_2$, $z=x_1+x_2$, and $w=x_1^2x_2$, we have
\[
J_2\bA^1=\Spec k[t,z,w]/(z^3-tzw+w^2).
\]
We may therefore view $J_2(\bA^1)$ as a one-parameter family of nodal cubics degenerating to a cusp when $t=0$. Geometrically, this family of cubics arises by looking at the lines $x_1+x_2=t$ in $\bA^2$; under the quotient map $q_2\colon \bA^2\to J_2\bA^1$, the points $(t,0)$ and $(0,t)$ are identified, leading to a nodal cubic.

Homogenizing with respect to $t$, one obtains
\[
J_2\bP^1=\Proj k[z,w][s,t]/(z^3-tzw+sw^2)
\]
where $\Proj$ is taken with respect to $s$ and $t$. Then at infinity, where $(s:t)=(1:0)$, we see our nodal cubic degenerates to a reducible curve $z(z^2-w)$.
\end{example}

We next prove that our James reduced product construction is functorial.

\begin{proposition}[{Functoriality of James reduced product}]
\label{prop:functorialityJnX}
Let $f\colon Y\to X$ be a morphism of FA-schemes over $R$. Let $\cI$ be an ideal sheaf on $X$ and equip $Y$ with the ideal sheaf $f^{-1}\cI$. Then we have an induced map $J_nf\colon J_nY\to J_nX$ making the diagram
\[
\xymatrix{
Y^n\ar[r]^-{f^n}\ar[d]_-{q_{n,Y}} & X^n\ar[d]^-{q_{n,X}}\\
J_nY\ar[r]^-{J_nf}\ar[d]_-{p_{n,Y}} & J_nX\ar[d]^-{p_{n,X}}\\
\Sym^n(Y)\ar[r]& \Sym^n(X)
}
\]
commute.
\end{proposition}
\begin{proof}
Let $X=\bigcup_i U_i$ be an open affine cover as in Lemma \ref{l:SymnX-scheme}. For each $i$, choose an open affine cover $f^{-1}(U_i)=\bigcup_j W_{ij}$ as in Lemma \ref{l:SymnX-scheme}. Let $U_i=\Spec A_i$ and $W_{ij}=\Spec B_{ij}$. By property (\ref{thm:OJnX--affine}) of Proposition \ref{prop:OJnX-sheaf},
\begin{equation}\label{eqn:property2affine1}
\Spec\QSym_{n,\cI(A_i)}(A_i)=J_n(U_i)\xrightarrow{\simeq} J_nX\times_{\Sym^n(X)}\Sym^n(U_i)
\end{equation}
and
\begin{equation}\label{eqn:property2affine2}
\Spec\QSym_{n,B_{ij}\cI(A)}(B_{ij})=J_n(W_{ij})\xrightarrow{\simeq} J_nY\times_{\Sym^n(Y)}\Sym^n(W_{ij}).
\end{equation}
By functoriality of the $\QSym$ construction in Definition \ref{def:qsym-ring-at-m}, we obtain morphisms $g_{ij}\colon J_n(W_{ij})\to J_n(U_i)$ commuting with the maps $q_{n,U_i}$, $p_{n,U_i}$, $q_{n,W_{ij}}$, and $p_{n,W_{ij}}$.

Next, recall that $X$ and $Y$ are separated (since they are FA-schemes), so for any $i,j,k,\ell$, we have that $U_{ik}\coloneqq U_i\cap U_k$ and $W_{ijk\ell}\coloneqq W_{ij}\cap W_{k\ell}$ are affine. Since $(f^n|_{W_{ij}})|_{W_{ijk\ell}}=(f^n|_{W_{k\ell}})|_{W_{ijk\ell}}$, since the $q_n$ maps induce injections on the level of rings, and since the $g_{\alpha\beta}$ maps commute with the relevant $q_n$ maps, it follows that the restriction of $g_{ij}$ to $\Sym^n(W_{ij})\cap\Sym^n(W_{k\ell})=\Sym^n(W_{ijk\ell})$ agrees with restriction of $g_{k\ell}$. As a result, then $g_{\alpha\beta}$ glue to give a map to $J_nX$.
\end{proof}

The next result shows that our James reduced product construction preserves open immersions.

\begin{proposition}\label{prop:JnopenImm}
Let $f\colon Y\to X$ be an open immersion with $X$ an FA-scheme over $R$. Let $\cI$ be an ideal sheaf on $X$ and equip $Y$ with the inverse image ideal sheaf $f^{-1}\cI$. Then
\[
J_nY\xrightarrow{\simeq} J_nX\times_{\Sym^n(X)}\Sym^n(Y).
\]
In particular, $J_nf\colon J_nY\to J_nX$ is an open immersion.
\end{proposition}
\begin{proof}
Since $X$ is an FA-scheme, $Y$ is as well. Let the $U_i$ and $W_{ij}$ be as in the proof of Proposition \ref{prop:functorialityJnX}. To show $J_nY\to J_nX\times_{\Sym^n(X)}\Sym^n(Y)$ is an isomorphism, it suffices to check this on the affine open cover $\Sym^n(W_{ij})$ of $\Sym^n(Y)$. Using equations \eqref{eqn:property2affine1} and \eqref{eqn:property2affine2}, we may replace $X$ (resp.~$Y$) by $U_i$ (resp.~$W_{ij}$), i.e., we may assume $X$ and $Y$ are affine. The result then follows from Lemma \ref{l:QSymOpenImmersions}.
\end{proof}

As a consequence we have

\begin{proposition}[{James reduced product of quasi-affine and quasi-projective schemes}]
\label{prop:JnX-qproj-scheme}
If $X$ is a quasi-projective (resp.~quasi-affine) scheme over $R$ equipped with an ideal sheaf $\cI$, then $J_nX$ is a quasi-projective (resp.~quasi-affine) scheme.
\end{proposition}
\begin{proof}
If $X$ is quasi-projective (resp.~quasi-affine), then we have an open immersion $f\colon X\to Y$ where $Y$ is projective (resp.~affine). Proposition \ref{prop:JnopenImm} then shows $J_nf\colon J_nX\to J_nY$ is an open immersion. By Proposition~\ref{prop:JnX-proj-scheme} (resp.\ Corollary~\ref{cor:JnX-affine-scheme}), $J_nY$ is projective (resp.\ affine), which proves the result.
\end{proof}

\subsection{Recovering the topological James reduced product}
\label{subsec:JnXan}

Our next objective is to show that our James reduced product construction recovers the usual topological James reduced product; hence, our construction endows the usual James reduced product of a complex variety with the structure of an algebraic variety. For this, we consider the following situation.

\begin{notation}\label{not:IdealIe}
Let $X$ be an FA-scheme over a field $k$, equipped with a $k$-valued base point $e\in X(k)$. Since $X\to\Spec k$ is separated, the section $e\colon\Spec k\to X$ is a closed immersion, and hence defines an ideal sheaf on $X$ which we denote by $\cI_e$.
\end{notation}

As in the previous subsection, we suppress all mention of our ground ring $R\coloneqq k$ in our notation for $\Sym$, $\QSym$, and $J_n$. For any map of affine schemes $\psi\colon\Spec C\to\Spec B$, we will let $\psi^*\colon B\to C$ be the induced map on rings.

We introduce the following definition which is akin to the equivalence relation used to define the topological James reduced product. We thank one of the referees for suggesting this definition.

\begin{definition}\label{def:simI}
Using the set-up in Notation \ref{not:IdealIe}, for any $k$-algebra $S$, let $e_S\in X(S)$ be the point induced by $e$. For any sequence $1\leq i_1<\dots<i_r\leq n$, let
\[
\varphi_{i_1,\dots,i_r}\colon X^r\to X^n
\]
be the morphism which on $S$-valued points sends $(a_1,\dots,a_r)$ to the sequence $(b_1,\dots,b_n)$ where $b_{i_j}=a_j$, and $b_h=e_S$ if $h\notin\{i_1,\dots,i_r\}$. Given two points $x,y\in X(S)^n$, we write $x\sim_e y$ if there exists $z\in X(S)^r$ and sequences $1\leq i_1<\dots<i_r\leq n$ and $1\leq j_1<\dots<j_r\leq n$ such that
\[
x=\varphi_{i_1,\dots,i_r}(z)\quad\textrm{and}\quad y=\varphi_{j_1,\dots,j_r}(z).
\]
In other words, $x\sim_e y$ if $x$ and $y$ are the same sequences up to inserting and deleting copies of $e_S$. We see $\sim_e$ defines an equivalence relation on $X(S)^n$.
\end{definition}

The following result shows that when $X$ is affine, any two $S$-valued points of $X^n$ which become equal in $J_nX$ must lie in the same equivalence class.

\begin{lemma}\label{l:one-direction-main-thm-for-all-R}
We keep the set-up in Notation \ref{not:IdealIe} and let $q_n\colon X^n\to J_nX$ be as in Proposition \ref{prop:OJnX-sheaf}. If $K/k$ is any field extension and $f,g\colon\Spec K\rightrightarrows X^n$ are maps over $k$, then
\[
f\sim_e g\quad\Longrightarrow\quad q_nf=q_ng.
\]
\end{lemma}
\begin{proof}
By definition, since $f\sim_e g$, we have a morphism $h\colon\Spec K\to X^r$ such that $f=\varphi_{i_1,\dots,i_r}h$ and $g=\varphi_{j_1,\dots,j_r}h$. The map $h$ corresponds to $K$-points $x_1,\dots,x_r$. Since $X$ is an FA-scheme, we can choose an open affine $U\subset X$ containing $x_1,\dots,x_r$ as well as $e_K$. Then, $h$ lifts to a morphism $h'\colon\Spec K\to U^r$. Since $U$ contains $e_K$, we have analogous maps $\varphi'_{i_1,\dots,i_r},\varphi'_{j_1,\dots,j_r}\colon U^r\rightrightarrows U^n$ and a commutative diagram
\[
\xymatrix{
& X^r\ar@<-.5ex>[rr]_-{\varphi_{j_1,\dots,j_r}}  \ar@<.5ex>[rr]^-{\varphi_{i_1,\dots,i_r}} & & X^n\ar[r]^-{q_{n,X}} & J_nX\\
\Spec K\ar[r]^-{h'}\ar[ur]^-{h} & U^r\ar[u]\ar@<-.5ex>[rr]_-{\varphi'_{j_1,\dots,j_r}}  \ar@<.5ex>[rr]^-{\varphi'_{i_1,\dots,i_r}} & & U^n\ar[u]\ar[r]^-{q_{n,U}} & J_nU\ar[u]
}
\]
Hence, we may replace $X$ by $U$ and assume $X$ is affine.

Let $X=\Spec A$, so that $e^*\colon A\to k$ is the map induced by $e\in X(k)$ and $I_e=\ker(e^*)$.
With notation as above, we have
\[
h^*\varphi^*_{i_1,\dots,i_r} = h^*\varphi^*_{j_1,\dots,j_r}
\]
and we want to show $f^*q_n^*=g^*q_n^*$. To do so, it is enough to show the two compositions
\[
\xymatrix{
\QSym_{n,I_e}(A)\ar[r]^-{q_n^*} & A^{\otimes n} \ar@<-.5ex>[rr]_-{\varphi^*_{j_1,\dots,j_r}}  \ar@<.5ex>[rr]^-{\varphi^*_{i_1,\dots,i_r}} & & A^{\otimes r}
}
\]
agree. This can be checked on generators of $\QSym_{n,I_e}(A)$. Let $f_1,\dots,f_s\in I_e$ and note that
\[
\varphi_{i_1,\dots,i_r}^*(a^{(i_j)})=a^{(j)}\quad\textrm{and\ otherwise}\quad\varphi^*_{i_1,\dots,i_r}(a^{(k)})=e^*(a).
\]
Since $e^*(f_j)=0$ for all $j$, we see that $\varphi_{i_1,\dots,i_r}^*(f_1^{(k_1)}\dots f_s^{(k_s)})=0$ if $k_1<\dots<k_s$ is not a subsequence of $i_1<\dots<i_r$, and
\[
\varphi_{i_1,\dots,i_r}^*(f_1^{(i_{\ell_1})}\dots f_s^{(i_{\ell_s})})=f_1^{(\ell_1)}\dots f_s^{(\ell_s)}.
\]
Thus,
\[
\varphi_{i_1,\dots,i_r}^*\bigg(\sum_{1\leq k_1<\dots<k_s\leq n}f_1^{(k_1)}\dots f_s^{(k_s)}\bigg)
=\sum_{1\leq\ell_1<\dots<\ell_s\leq r}f_1^{(\ell_1)}\dots f_s^{(\ell_s)}.
\]
This finishes the proof since the above expression is independent of $i_1<\dots<i_r$.
\end{proof}

The following lemma relates equivalence classes of $X(S)^n$ with those of $X(K)^n$ when $S$ is an integral domain and $K$ is its fraction field. This will play a role in Theorem \ref{thm:functor-of-pts-JnX} below.

\begin{lemma}\label{l:sim-RvsK}
We keep the set-up in Notation \ref{not:IdealIe} and assume $X$ is locally finitely presented over $k$. Let $S$ be an integral domain with fraction field $K$, let $f,g\in X(S)^n$ and $f_K,g_K\in X(K)^n$ be the induced points. Then $f\sim_e g$ if and only if $f_K\sim_e g_K$.
\end{lemma}
\begin{proof}
It follows immediately from the definition that if $f\sim_e g$, then $f_K\sim_e g_K$. So it remains to prove the converse. By definition, we have a commutative diagram
\[
\xymatrix{
\Spec K\ar[r]^-{h}\ar[d]_-{\iota} & X^r\ar[d]^-{\varphi_{i_1,\dots,i_r}}\\
\Spec S\ar[r]^-{f} & X^n
}
\]
and a similar commutative diagram where we replace $f$ (resp.~$\varphi_{i_1,\dots,i_r}$) by $g$ (resp.~$\varphi_{j_1,\dots,j_r}$). Since $X$ is locally finitely presented over $k$, we may extend the domain of $h$ to a dense open affine $U\coloneqq \Spec S'\subset\Spec S$ where $S\subset S'\subset K$ and where we still maintain commutativity of the two aforementioned diagrams; by abuse of notation, we will continue to denote the extended map $U\to X^r$ by $h$ and the map $U\to\Spec S$ by $\iota$.

Since $e\colon\Spec k\to X$ is a closed immersion, $\varphi_{i_1,\dots,i_r}$ is as well. Letting $Z=X^r\times_{X^n}\Spec S$, we see $\iota$ factors as $U\to Z\xrightarrow{\jmath}\Spec S$, where $\jmath$ is a closed immersion. On the level of rings, we have $S\twoheadrightarrow H^0(\cO_Z)\to S'$; since the composition is injective, $\jmath^*$ is an isomorphism, i.e., $\jmath$ is an isomorphism. Thus, there is a unique map $h'\colon\Spec S\to X^r$ such that
\[
h=h'\iota\quad\textrm{and}\quad f=\varphi_{i_1,\dots,i_r}h'.
\]
Similarly, there is a unique map $h''\colon\Spec S\to X^r$ such that
\[
h=h''\iota\quad\textrm{and}\quad g=\varphi_{j_1,\dots,j_r}h''.
\]
Note that $h'|_U=h=h''|_U$. Since $\Spec S$ is reduced and $X^r$ is separated (as $X$ is an FA-scheme), we must have $h'=h''$. In particular, $f\sim_e g$.
\end{proof}

We are now ready to state the following theorem.

\begin{theorem}\label{thm:functor-of-pts-JnX}
We use the set-up in Notation \ref{not:IdealIe}, assume $X$ is finitely presented over $k$, and let $q_n\colon X^n\to J_nX$ be as in Proposition \ref{prop:OJnX-sheaf}. If $S$ is an integral domain and $f,g\colon\Spec S\rightrightarrows X$ are maps, then
\[
q_nf=q_ng\quad\textrm{if\ and\ only\ if}\quad f\sim_e g.\]
In particular,
\[
(J_nX)(K)=X(K)^n/\sim_e
\]
for any algebraically closed field $K$.
\end{theorem}

Before proving Theorem \ref{thm:functor-of-pts-JnX}, we note that the integrality hypothesis on $S$ cannot be dropped.

\begin{example}\label{ex:need-integrality}
Recall that $S$ is integral if and only if $\Spec S$ is reduced and irreducible. We show that if one drops either of these conditions, then the conclusion of Theorem \ref{thm:functor-of-pts-JnX} is false. Let $k$ be a field, $X=\bA^1_k$ equipped with the base point $0$, and recall from Example \ref{ex:J2A1} that $\QSym_{2,(x)}k[x]\subset k[x_1,x_2]$ is the ring generated by $x_1+x_2$, $x_1x_2$, and $x_1^2x_2$.

We first give an example where $\Spec S$ is non-reduced. Consider the maps
\[
f,g\colon k[x_1,x_2]\rightrightarrows S\coloneqq k[y]/y^3
\]
given by
\[
f(x_1)=y,\quad f(x_2)=y^2,\quad g(x_1)=y^2,\quad\textrm{and}\quad g(x_2)=y.
\]
Then $f$ and $g$ agree on $\QSym_{2,(x)}k[x]$ but are not in the same equivalence class.

We next give an example where $\Spec S$ is reducible. Consider
\[
f,g\colon k[x_1,x_2]\rightrightarrows S \coloneqq k[y,z]/(yz(y-z))
\]
with
\[
f(x_1)=y,\quad f(x_2)=z,\quad g(x_1)=z,\quad\textrm{and}\quad g(x_2)=y.
\]
Then we see $y^2z=z^2y$ by definition, so $f$ and $g$ agree on $\QSym_{2,(x)}k[x]$ but are again not in the same equivalence class.
\end{example}

\begin{remark}
One may wonder if the functor of points of $J_nX$ is given by $F(S) \coloneqq X(S)^n/\sim_e$, which is in the same spirit as the topological definition of the James reduced product. There are two main issues here. The first is that, as shown in Example \ref{ex:need-integrality}, the integrality hypothesis on $S$ is necessary. The second issue is that $S$-points of $J_nX$ do not necessarily lift to $S$-points of $X^n$. If instead one considers the \emph{qfh-topology}, introduced by Voevodsky \cite[Definition 3.1.2]{VoevodskyHomology}, then $S$-points of $J_nX$ locally lift to $X^n$. This, together with Theorem \ref{thm:functor-of-pts-JnX}, suggests that the qfh-sheaf associated to $J_nX$ agrees with the qfh-sheafification of $F$; it would be interesting to verify the equality of these two qfh-sheaves. Note, however, that knowing the qfh-sheaf associated to a scheme $Y$ does not recover $Y$, but only its topological type, see \cite[Proposition 3.2.5]{VoevodskyHomology}.
\end{remark}

\begin{proof}[{Proof of Theorem \ref{thm:functor-of-pts-JnX}}]
We begin by reducing to the case where $S$ is a field. Since $X$ is separated, locally finitely presented, and flat over $k$, combining Theorem (3.4.1), Proposition (4.3.1)(iv), and Corollary (4.2.5) of \cite{Rydh.dividedpowers}
shows that $\Sym^n(X)$
is separated and locally finitely presentated over $k$. Since $J_nX\to\Sym^n(X)$ is finite, $J_nX$ is separated and locally finitely presented over $k$ as well.

For any map $V\to\Spec S$, let $f_V$ and $g_V$ denote the induced maps $V\to X^n$. Since $J_nX$ is finitely presented over $k$, we see $q_nf_K=q_ng_K$ if and only if there exists a dense open subset $V\subset\Spec S$ such that $q_nf_V=q_ng_V$. As $\Spec S$ is reduced and $J_nX$ is separated, $q_nf_V=q_ng_V$ if and only if $q_nf=q_ng$. We have therefore shown that $q_nf=q_ng$ if and only if $q_nf_K=q_ng_K$. Assuming the statement of the theorem for fields, the latter equality is equivalent to $f_K\sim_e g_K$, which by Lemma \ref{l:sim-RvsK}, is equivalent to $f\sim_e g$.

We therefore reduced to the case where $S=K$ is a field. In light of Lemma \ref{l:one-direction-main-thm-for-all-R}, it remains to prove that if $fq_n=gq_n$, then $f\sim_e g$. The map $f$ (resp.~$g$) corresponds to points $p_1,\dots,p_n\in X(K)$ (resp.~$p'_1,\dots,p'_n\in X(K)$). Since $X$ is an FA-scheme, we may choose an open affine $\iota\colon U\hookrightarrow X$ containing all $p_i$, all $p'_i$, and also $e_K$. Then $f$ and $g$ factor through $U^n$; denote these factorizations as $f'$ and $g'$. By hypothesis $J_n(\iota)\circ q_{n,U}f'=q_{n,X}f=q_{n,X}g=J_n(\iota)\circ q_{n,U}g'$. By Proposition \ref{prop:JnopenImm}, $J_n(\iota)$ is an open immersion, hence $q_{n,U}f'=q_{n,U}g'$. We have therefore further reduced to the case where $X$ is affine.

We may now assume $X=\Spec A$. Keeping the notation from the previous paragraph, $f=(p_1,\dots,p_n)$ and $g=(p'_1,\dots,p'_n)$. Let $1\leq i_1<\dots<i_r\leq n$ such that the $p_{i_\ell}$ are non-base points and all other $p_\ell=e_K$. Let $1\leq j_1<\dots<j_s\leq n$ such that $p'_{j_\ell}$ are non-base points and all other $p'_\ell=e_K$. We first show $r=s$. If this is not the case, then without loss of generality, we may assume $r<s$. Let $f_1,\dots,f_s\in A$ such that $f_\ell(p'_{j_\ell})\neq0$ and $f_\ell(e)=0$ (so all $f_\ell\in I_e$). Then
\[
F \coloneqq \sum_{k_1<\dots<k_s}f_1^{(k_1)}\dots f_s^{(k_s)}\in \QSym_{n,I_e}(A)
\]
satisfies
\[
g^*q_n^*(F)=\prod_{\ell=1}^s f_\ell(p'_{j_\ell})\neq0\quad\textrm{and}\quad f^*q_n^*(F)=0.
\]
Thus, we must have $r=s$.

Having now shown that $f$ and $g$ have the same number of base points, after replacing $f$ and $g$ by elements in the same $\sim_e$ equivalence class, we may assume $f=(p_1,\dots,p_s,e_K,\dots,e_K)$ and $g=(p'_1,\dots,p'_s,e_K,\dots,e_K)$ where all $p_i$ and $p'_i$ are non-base points. We claim that $f=g$. If this is not the case, then there exists an index $m$ such that $p_m\neq p'_m$. Choose $f_1,\dots,f_s\in A$ such that following hold: $f_\ell(e_K)=0$ for all $\ell$; $f_m(p_m)=0$; $f_m(p'_m)\neq 0$; and $f_\ell(p_\ell),f_\ell(p'_\ell)\neq 0$ for all $\ell\neq m$. Letting $F$ be as above, we again have $g^*q_n^*(F)=\prod_{\ell=1}^s f_\ell(p'_\ell)\neq0$ and $f^*q_n^*(F)=0$. We must therefore have $f=g$.
\end{proof}

\vspace{1em}

Using Theorem \ref{thm:functor-of-pts-JnX}, we can now show that our James reduced product construction recovers the usual James reduced product of a complex variety.

\begin{theorem}[{Recovering topological James reduced product}]
\label{thm:JnXan}
Let $X$ be a quasi-projective variety over $\C$ equipped with a base point $e\in X(\C)$, and let $\cI_e$ be the ideal sheaf of the base point. Then
\[
(J_nX)^{\an}=J_n(X^{\an}),
\]
where the superscript $\an$ denotes the analytification.
\end{theorem}
\begin{proof}
By Theorem \ref{thm:functor-of-pts-JnX}, we see $(J_nX)(\C)=X(\C)^n/\!\sim_e\ =X(\C)^n/\sim$, where $\sim$ is the equivalence relation used to define the topological James reduced product of $X(\C)$. Thus, $q_n^{\an}\colon (X^{\an})^n\to (J_nX)^{\an}$ endows $(J_nX)^{\an}$ with the quotient topology with respect to $\sim$, thereby proving the result.
\end{proof}

\section{Double quasisymmetric functions and equivariant cohomology of James reduced products}\label{sec:cohomology}

We use the following notation and terminology throughout the rest of the paper. A \newword{weak composition} is a finite sequence of nonnegative integers and a \newword{composition} is a finite sequence of positive integers. If $\alpha$ is a weak composition, its \newword{positive part} $\alpha^+$ is the composition obtained from $\alpha$ by deleting all $0$ terms, e.g., if $\alpha=(1,7,0,0,5,0,5)$, then $\alpha^+=(1,7,5,5)$. Given positive integers $n < n'$, we frequently write $[n] = \{ 1, 2, \dots, n\}$ and $[n,n'] = \{n, n+1, \dots, n'\}$. Let $\cC$ denote the set of all compositions, and let $\cC_{n,m}$ denote the set of compositions of length at most $n$ with entries at most $m$.

We have a sequence of inclusions $\C\bP^0\subset\C\bP^1\subset\cdots$ with $\C\bP^\infty=\bigcup_{n\geq0}\C\bP^n$. Then $\C\bP^\infty$ has the structure of a CW complex where its cells are given by $e_{2n}=\C\bP^n\setminus\C\bP^{n-1}\simeq\bA^n_\C$. Furthermore, each $\C\bP^n$
carries an action of the algebraic torus $T_{n+1} \coloneqq (\C^*)^{n+1}$ given by $(\lambda_1,\dots,\lambda_{n+1})\cdot(x_0:\ldots:x_n)=(\lambda_1x_0:\ldots:\lambda_{n+1}x_n)$.
Letting $T=\lim_n T_n=(\C^*)^\infty$, we obtain a $T$-action on $\C\bP^\infty$. Notice that each cell $e_{2i}$ is $T$-invariant, and that the $T$-action on $\C\bP^n$ is induced by the surjection $T\twoheadrightarrow T_{n+1}$ and the $T_{n+1}$-action on $\C\bP^n$.

By definition, the CW complex structure on $J_n\C\bP^\infty$ is induced from the cellular quotient map
\[
q_n\colon (\C\bP^\infty)^n\to J_n\C\bP^\infty.
\]
By construction, $q_n$ identifies the cells $e_{2b_1}\times\dots\times e_{2b_n}$ and $e_{2c_1}\times\dots\times e_{2c_n}$ of $(\C\bP^\infty)^n$ if and only if $(b_1,\dots,b_n)^+ = (c_1,\dots,c_n)^+$. As a result, the cells of $J_n\C\bP^\infty$ are indexed by compositions $\alpha=(\alpha_1,\dots,\alpha_k)$ with length $k\leq n$ and the cells of $J\C\bP^\infty$ are indexed by the set $\cC$ of all compositions. For any composition $\alpha$, we let
\[
e_\alpha\subset J\C\bP^\infty
\]
denote the corresponding cell. Furthermore, the diagonal $T$-action on $(\C\bP^\infty)^n$ descends to a $T$-action on $J_n\C\bP^\infty$ so that each $e_\alpha$ is a $T$-invariant cell that contains a unique $T$-fixed point.

We next prove a general result on torus-equivariant cohomology of projective varieties with affine pavings.

\begin{proposition}\label{prop:base-change-HT}
Let $Y$ be a complex projective variety with an action of a torus $\TT$. Suppose $Y$ has an affine paving by $\TT$-invariant cells $U_1,\dots,U_N\subset Y$ and that each $U_i\simeq\bA^{n_i}$ has a unique $\TT$-fixed point $c_i$. Then
\begin{enumerate}
\item\label{general-prop-injection} we have an injection
\[
H^*_{\TT}(Y; \Z)\hookrightarrow H^*_{\TT}(Y^{\TT}; \Z)\simeq \bigoplus_{i=1}^N H^*_\TT(pt; \Z),
\]

\item\label{general-prop-basis} the equivariant classes $[U_i]_\TT$ form a basis for $H^*_\TT(Y;\Z)$ as an $H^*_\TT(pt)$-module, and

\item\label{general-prop-base-change} If $\TT'\to\TT$ is a morphism of tori, then
\[
H^*_{\TT'}(Y; \Z)=H^*_{\TT'}(pt; \Z)\otimes_{H^*_{\TT}(pt; \Z)} H^*_{\TT}(Y; \Z).
\]
\end{enumerate}
\end{proposition}
\begin{proof}
Since $Y$ has a cell decomposition consisting of even-dimensional $\TT$-invariant cells, its equivariant Borel--Moore homology $\overline{H}^\TT_*(Y;\Z)$ is free with basis given by the cell closures, see e.g., \cite[Proposition~17.1.2]{AndersonFulton}. Since $H^*_\TT(Y;\Z)$ is the dual of $\overline{H}^\TT_*(Y;\Z)$ over $H^*_\TT(pt)$, statement (\ref{general-prop-basis}) follows, see \cite[Proposition~2.1(b)]{Graham}. By construction, this basis is compatible with change of groups, so (\ref{general-prop-base-change}) holds. Lastly, (\ref{general-prop-injection}) follows from the localization package, see \cite[Theorem 17.3.1]{AndersonFulton}.
\end{proof}

Let $\Lambda=H^*_T(pt; \Z)=\Z[y_1,y_2,\dots]$ and $\Lambda_m=H^*_{T_{m}}(pt; \Z)=\Z[y_1,\dots,y_m]$. Note that for the natural gradings on $\Lambda$ and $\Lambda_m$, the $y_i$ have degree $2$.
Applying Proposition \ref{prop:base-change-HT} to $J_n\C\bP^m$ with $\TT=T_{m+1}$ and $\TT'=T$, and using the fact that $J_n\C\bP^m$ is a complex projective variety by Proposition~\ref{prop:JnX-proj-scheme}, we have an injection
\begin{equation}\label{eqn:injHTFixedPt}
\iota_{n,m}\colon H^*_{T_{m+1}}(J_n\C\bP^m; \Z)\hookrightarrow H^*_{T_{m+1}}((J_n\C\bP^m)^{T_{m+1}}; \Z)\simeq \bigoplus_{\alpha\in\cC_{n,m}} \Lambda_{m+1},
\end{equation}
as well as an isomorphism
\[
H^*_{T_{m+1}}(J_n\C\bP^m) \simeq \bigoplus_{\alpha\in\cC_{n,m}}\Lambda_{m+1} x_\alpha
\]
of graded $\Lambda_{m+1}$-modules, where $x_\alpha$ denotes the class associated to the $T$-invariant cell $e_\alpha$. Furthermore, we have the following.

\begin{corollary}\label{cor:limit-JCPinfty}
We have
\[
H^*_T(J\C\bP^\infty)\simeq\lim_{n,m\to\infty}\Lambda\otimes_{\Lambda_{m+1}} H^*_{T_{m+1}}(J_n\C\bP^m),
\]
where the limit is in the category of graded $\Lambda$-algebras. In particular,
\[
H^*_T(J\C\bP^\infty)=\bigoplus_{\alpha\in\cC}\Lambda x_\alpha.
\]
\end{corollary}
\begin{proof}
Since the $T$-action on $J_n\C\bP^m$ is induced by the surjection $T\twoheadrightarrow T_{m+1}$, Proposition \ref{prop:base-change-HT}(\ref{general-prop-base-change}) shows that $H^*_T(J_n\C\bP^m)=\Lambda\otimes_{\Lambda_{m+1}}H^*_{T_{m+1}}(J_n\C\bP^m)$. It therefore suffices to prove that
\[
H^*_T(J\C\bP^\infty)\simeq\lim_{n,m\to\infty}H^*_{T}(J_n\C\bP^m).
\]
We have a diagram
\[
\xymatrix{
J_n\C\bP^m \times ET\ \ar@{^{(}->}[r]\ar@{->>}[d]_-{\pi_{n,m}}  & J\C\bP^\infty \times ET\ar@{->>}[d]^-{\pi}\\
(J_n\C\bP^m)_T\ar[r] & (J\C\bP^\infty)_T
}
\]
of mixing spaces.
It is straightforward to check that the bottom horizontal map is an inclusion and that the diagram is cartesian. It follows that $(J_n\C\bP^m)_T\subset (J\C\bP^\infty)_T$ is closed. Moreover, since $J\C\bP^\infty=\bigcup_{n,m}J_n\C\bP^m$, we see that $(J\C\bP^\infty)_T=\bigcup_{n,m}(J_n\C\bP^m)_T$.

For any fixed $k$, every $k$-cell of $J\C\bP^\infty \times ET$ is a $k$-cell of $J_n\C\bP^m \times ET$ for $n,m\gg0$. Since the quotient maps $\pi_{n,m}$ and $\pi$ respect the $T$-invariant cell structures on $J_n\C\bP^m \times ET$ and $J\C\bP^\infty \times ET$, it follows that for any fixed $k$, every $k$-cell of $(J\C\bP^\infty)_T$ is a $k$-cell of $(J_n\C\bP^m)_T$ for $n,m\gg0$. Thus, the map
\[
H^k_T(J\C\bP^\infty;\Z)=H^k((J\C\bP^\infty)_T;\Z)\xrightarrow{\simeq} H^k((J_n\C\bP^m)_T;\Z)=H^k_T(J_n\C\bP^m;\Z)
\]
is an isomorphism for $n,m\gg 0$.
\end{proof}

We now turn to the proof of our second main result, Theorem~\ref{james-reduced-cells->monomial-basis}.

\begin{proof}[{Proof of Theorem \ref{james-reduced-cells->monomial-basis}}]
We have
\[
H^*_{T_{m+1}}(\C\bP^m; \Z)=\Lambda_{m+1}[x]/\prod_{j=1}^{m+1}(x-y_j),
\]
where $x$ is the equivariant Chern class $c_1^T(\cO(-1))$, see e.g., \cite[\S 2.6]{AndersonFulton}. Then
\[
H^*_{T_{m+1}}((\C\bP^m)^n; \Z)=H^*_{T_{m+1}}(\C\bP^m; \Z)^{\otimes n}=\Lambda_{m+1}[x_1,\dots,x_n]/(\prod_{j=1}^{m+1}(x_i-y_j)),
\]
where $x_i=c_1^T(\cO(0,\dots,0,-1,0,\dots,0))$ with the non-zero entry in the $i$th position; furthermore, the Poincar\'e dual to the equivariant homology class of the closure of the cell $e_{2a_1}\times\dots\times e_{2a_n}$ is given by
\[
[e_{2a_1}\times\dots\times e_{2a_n}]_{T_{m+1}}=\prod_{i=1}^n\prod_{j=1}^{a_i}(x_i-y_j).
\]

Let $q_{n,m}\colon (\C\bP^m)^n\to J_n\C\bP^m$ be the quotient map. We claim that the induced map $q_{n,m}^*$ on $T_{m+1}$-equivariant cohomology is injective. To see this, recall we have an injection \eqref{eqn:injHTFixedPt} defined above. We also have a similar injection
\[
\iota'_{n,m}\colon H^*_{T_{m+1}}((\C\bP^m)^n; \Z)\hookrightarrow H^*_{T_{m+1}}(((\C\bP^m)^n)^{T_{m+1}}; \Z)\simeq \bigoplus_{[0,m]^n}\Lambda_{m+1}.
\]
By functoriality, we have a commutative diagram
\[
\xymatrix{
H^*_{T_{m+1}}((\C\bP^m)^n; \Z)\ar@{^{(}->}[r]^-{\iota'_{n,m}} & \bigoplus_{[0,m]^n}\Lambda_{m+1}\\
H^*_{T_{m+1}}(J_n\C\bP^m; \Z)\ar[u]^-{q_{n,m}^*}\ar@{^{(}->}[r]^-{\iota_{n,m}} & \bigoplus_{\alpha\in\cC_{n,m}} \Lambda_{m+1}\ar[u]^-{p^*_{n,m}}
}
\]
where $p_{n,m}\colon ((\C\bP^m)^n)^{T_{m+1}}\to (J_n\C\bP^m)^{T_{m+1}}$ is the induced map on the level of fixed points. Now note that $p_{n,m}$ has a splitting that sends the unique fixed point in the cell $e_{(\alpha_1,\dots,\alpha_k)}$ to the unique fixed point in the cell $e_{2\alpha_1}\times\dots\times e_{2\alpha_k}\times e_0\times\dots\times e_0$. Thus, $p_{n,m}^*$ is injective, and hence by commutativity of the above diagram, $q_{n,m}^*$ is injective as well.

We now explicitly compute the image of the injection $q_{n,m}^*$. If $\alpha=(\alpha_1,\dots,\alpha_k)\in\cC_{n,m}$, then
\[
q_{n,m}^{-1}(e_\alpha)=\coprod_{\iota} e_\iota,
\]
where the coproduct is over weak compositions $\iota=(i_1,\dots, i_n)\in[0,m]^n$ such that $(i_1,\dots, i_n)^+=\alpha$, and $e_\iota \coloneqq e_{2i_1}\times\dots\times e_{2i_n}$.
As a result, we have
\[
q_{n,m}^*(x_{(\alpha_1,\dots,\alpha_k)})=\sum_\iota [e_\iota]_{T_{m+1}}=\sum_{1\leq i_1<\dots<i_k\leq n}\prod_{\ell=1}^k\prod_{j_\ell=1}^{\alpha_\ell}(x_{i_\ell}-y_{j_\ell}) \eqqcolon M_{n,m+1,(\alpha_1,\dots,\alpha_k)}(\bx,\by);
\]
notice that $M_{n,m+1,(\alpha_1,\dots,\alpha_k)}(\bx,\by)$ is the double monomial quasisymmetric function $M_{(\alpha_1,\dots,\alpha_k)}(\bx,\by)$ as introduced in \eqref{eq:double_monomial}, truncated to the finite set of variables $x_1,\dots,x_n,y_1,\dots,y_{m+1}$.

Next, letting $\jmath_n\colon(\C\bP^m)^{n-1}\to(\C\bP^m)^n$ be the map $\id\times e$, we claim
\[
\jmath_n^*\colon H^*_{T_{m+1}}((\C\bP^m)^n; \Z)\to H^*_{T_{m+1}}((\C\bP^m)^{n-1}; \Z)
\]
is given by
\[
\jmath_n^*(x_i)=
\begin{cases}
x_i,& i<n;\\
y_1,& i=n.
\end{cases}
\]
To see this, let $\textrm{pr}_{r,i}\colon(\C\bP^m)^r\to\C\bP^m$ be the projection onto the $i$th factor. Then for $i<n$, we have $\textrm{pr}_{n,i}\circ\jmath_n=\textrm{pr}_{n-1,i}$, so
\[
\jmath_n^*(x_i)=\jmath_n^*c_1^T(\textrm{pr}_{n,i}^*\cO(-1))=c_1^T(\textrm{pr}_{n-1,i}^*\cO(-1))=x_i;
\]
on the other hand, $\textrm{pr}_{n,n}\circ\jmath_n$ is the constant map sending all points to $e$. Thus,
\[
\jmath_n^*(x_n)=c_1^T(\cO(-1)_e)=y_1.
\]

Applying Corollary \ref{cor:limit-JCPinfty} and taking limits in the category of graded algebras over the graded ring $\Lambda$, we obtain an injection
\[
H^*_T(J\C\bP^\infty;\Z)\hookrightarrow \lim_{n,m\to\infty} \Lambda[x_1,\dots,x_n]/(\prod_{j=1}^{m+1}(x_i-y_j))=\Lambda[[x_1-y_1,x_2-y_1,\dots]]_{\textrm{gr}}
\]
where $\Lambda[[x_1-y_1,x_2-y_1,\dots]]_{\textrm{gr}}$ denotes graded power series.
The image of the class $x_\alpha$ under this injection must be $M_\alpha(\bx,\by)$ since for all $n$ and $m$ sufficiently large, $x_\alpha\in H^*_T(J(\C\bP^\infty),\Z)$ maps to the truncation $M_{n,m+1,\alpha}(\bx,\by)$ in $\Lambda_{m+1}[x_1,\dots,x_n]/(\prod_{j=1}^{m+1}(x_i-y_j))$.
Lastly, by graded Nakayama, the $\Lambda$-submodule of $\Lambda[[x_1,x_2,\dots]]_{\textrm{gr}}$ generated by the $M_\alpha(\bx,\by)$ is $\QSym(\Lambda)$, yielding the isomorphism $H^*_T(J\C\bP^\infty;\Z)\simeq\QSym(\Lambda)$.
\end{proof}

\begin{example}\label{rmk:doubleM-are-in-appropriate-power-series}
It is clarifying to explicitly see from the definition (i.e., equation \eqref{eq:double_monomial}) that the $M_\alpha(x,y)$ are indeed in $\Lambda[[x_1-y_1,x_2-y_1,\dots]]_{\textrm{gr}}$. Rewriting $x_{i_\ell}-y_{j_\ell}$ as the sum of $x_{i_\ell}-y_1$ and $y_1-y_{j_\ell}$, we see
\[
\prod_{1\leq j_\ell\leq a_\ell}(x_{i_\ell}-y_{j_\ell})=\sum_{m=1}^{a_\ell}(x_{i_\ell}-y_1)^m\, e_{m-1}(y_1-y_2,\dots,y_1-y_{a_\ell})
\]
where $e_m$ denotes the $m$th elementary symmetric function. Thus, letting $A=\max_\ell\{a_\ell\}$, we see
\[
M_\alpha(\bx,\by)=\sum_{j=1}^k\sum_{1\leq m_j\leq a_j}P_{m_1,\dots,m_k}(y)\sum_{1\leq i_1<\dots<i_k}\prod_{\ell=1}^k(x_{i_\ell}-y_1)^{m_\ell}
\]
where $P_{m_1,\dots,m_k}(y)$ is a polynomial in $y_1-y_2,\dots,y_1-y_A$. Thus, $M_\alpha(\bx,\by)\in\Lambda[[x_1-y_1,x_2-y_1,\dots]]_{\textrm{gr}}$.
\end{example}

\begin{remark}
Using the scheme-theoretic definitions of Section \ref{sec:variety}, one can replace $\C\bP^\infty$ with $\bP^\infty_k$ for any field $k$ and ask for an analogue of Theorem \ref{james-reduced-cells->monomial-basis} using equivariant Chow groups in place of equivariant cohomology, see \cite{Edidin.Graham,Brion97}. For such an extension of the classical Littlewood--Richardson rules see \cite{Vakil06}.
\end{remark}

\begin{remark}\label{rmk:ChangSkjelbred}
It is worth remarking that standard tools introduced in \cite{GKM} and \cite{ChangSkjelbred} do not appear to imply Theorem \ref{james-reduced-cells->monomial-basis}.
Indeed, since $X:=J_n\bP^m$ has infinitely many $T$-fixed curves, the results of \cite{GKM}, as stated, do not apply. Lemma 2.3 of \cite{ChangSkjelbred}, on the other hand, allows for infinitely many $T$-fixed curves, yielding an exact sequence
\[
0\to H^*_T(X)\to H_T^*(X^T)\xrightarrow{\delta} H_T^{*+1}(X_1,X^T),
\]
where $X_1$ denotes the $1$-skeleton. We are given elements in $H_T^*(X^T)$ coming from specializing our $M_\alpha(\bx,\by)$, which we wish to show live in $H^*_T(X)$ and form a basis. We must therefore show our elements map to $0$ under $\delta$. By \cite[Remark 1.2]{FranzCS}, this is equivalent to showing that our elements of $H_T^*(X^T)$ are in the image of every $H_T^*(X^K)$ where $K\subset T$ runs through all codimension-one subtori that occur as the identity components of an isotropy group. Specifically, such $K$ are given by $K=(\C^*)^m\to(\C^*)^{m+1}=T$ mapping via $(t_1,\dots,t_m)\mapsto (t_1,\dots,t_{i-1},t_i,t_i,t_{i+1},\dots,t_m)$. One calculates then that $(J_n\bP^m)^K$ is the union of $J_n\bP^1$ and the $T$-fixed points. Thus, one can reduce to the case when $m=1$. However, when $m=1$, the sequence gives no further information:~it is simply the tautological statement that our elements of $H_T^*(X^T)$ live in $H^*_T(X)$ if and only if they are in the image of the identity map $\id\colon H^*_T(X^T)\to H_T^*(X^T)$. Hence, \cite{ChangSkjelbred} does not appears to be sufficient to handle the case $m=1$.
\end{remark}

\section{Double quasisymmetric analogue of the Littlewood--Richardson rule}\label{sec:LR}

In this section, we give a cancellation-free combinatorial formula for the product $M_\alpha(\bx,\by)\cdot M_\beta(\bx,\by)$ of double monomial quasisymmetric functions. For the specializations to ordinary monomial quasisymmetric functions $M_\alpha(\bx) \coloneqq M_\alpha(\bx,0)$ such a rule appears in work of M.~Hazewinkel \cite{Hazewinkel}, and our rule extends his into the equivariant setting.

To give our rule, we first recall work of Thomas--Yong \cite{Thomas.Yong18} on equivariant Schubert calculus. (Their results hold more generally for Grassmannians, but we will only need the specializations of their work to projective spaces.) Let $c > a \in \Z_{\geq 0}$. Then a \newword{skew edge-labeled tableau} $S$ of shape $c/a$ consists of a row of $c$ boxes such that:
\begin{itemize}
\item the rightmost $c-a$ boxes of $S$ are filled with a $\bullet$ label,
\item the remaining $a$ boxes are empty, and
\item the lower horizontal edges of the leftmost $a$ boxes may optionally and independently carry an edge label of $\bullet$.
\end{itemize}
Letting $b$ be the total number of $\bullet$s on boxes and edges of $S$, we say $S$ has \newword{content} $b$. Note that $c-a \leq b \leq c$. To each such $S$, Thomas--Yong \cite{Thomas.Yong18} associate a weight $\wt(S)\in\Lambda$ as follows:~if one of the last $c-a$ boxes is unlabeled, then $\wt(S)=0$. Otherwise, let $E \subseteq [a]$ be the set with $i\in E$ if and only if the $i$-th box of $S$ carries an edge label; then
\[
\wt(S)=\prod_{i\in E}(y_i - y_{i+1+r(i)}),
\]
where $r(i)$ is the number of (box or edge) labels appearing strictly right of the $i$-th box. Note that if $\wt(S)\neq0$, then $c\leq a+b$.

\begin{example}
Let $c = 7$ and $a = 4$. Then
	\[\Scale[0.9]{\begin{picture}(140,33)
\ytableausetup{boxsize=2.0em}
\put(-20,6) {$S=$}
\put(0,0){$\ytableaushort{ \blank \blank \blank \blank \bullet \bullet \bullet}$}
\put(8,-3){$\bullet$}
\put(49,-3){$\bullet$}
\end{picture}}
\quad \text{and} \quad
 \Scale[0.9]{\begin{picture}(140,33)
\ytableausetup{boxsize=2.0em}
\put(0,6) {$S'=$}
\put(23,0){$\ytableaushort{ \blank \blank \blank \blank \bullet \bullet \bullet}$}
\put(31,-3){$\bullet$}
\put(52,-3){$\bullet$}
\put(92,-3){$\bullet$}
\end{picture}}
\]
are both skew edge-labeled tableaux of shape $c/a$. The content of $S$ is $5$, while the content of $S'$ is $6$. Moreover,
\[
\wt(S) = (y_1 - y_{1+1+4})  (y_3 - y_{3+1+3}) = (y_1-y_6)(y_3-y_7),
\]
while
\[
\wt(S') = (y_1 - y_{1+1+5})  (y_2 - y_{2+1+4}) (y_4 - y_{4+1+3}) = (y_1-y_7)(y_2-y_7)(y_4-y_8). \qedhere
\]
\end{example}

An argument analogous to the proof of Corollary \ref{cor:limit-JCPinfty} shows that
\[
H^*_T(\C\bP^\infty;\Z)=\lim_{m\to\infty}\Lambda\otimes_{\Lambda_{m+1}} H^*_{T_{m+1}}(\C\bP^m;\Z).
\]
Then by \cite[Theorem 1.2]{Thomas.Yong18}, we have the following equivariant Littlewood--Richardson rule for infinite projective space.
We let $x_k\in H^*_T(\C\bP^\infty;\Z)$ denote the equivariant class of the $2k$-cell $e_{2k}$.

\begin{theorem}[{\cite[Theorem 1.2]{Thomas.Yong18}}]\label{thm:Thomas-Yong}
In $H^*_T(\C\bP^\infty;\Z)$, we have
\[
x_a\cdot x_b=\sum_c \sum_S \wt(S) x_c
\]
where $S$ runs over all skew edge-labeled tableaux of shape $c/a$ and content $b$.
\end{theorem}

We introduce the following notation, related to the \emph{overlapping shuffles} of Hazewinkel \cite{Hazewinkel}.

\begin{definition}
Let $\alpha=(\alpha_1, \dots, \alpha_\ell)$, $\beta = (\beta_1, \dots, \beta_m)$, and $\gamma = (\gamma_1, \dots, \gamma_n)$ be compositions. Let $\iota : [\ell] \to [n]$ and $\jmath : [m] \to [n]$ be order-preserving injections. Let $\mathcal{S}_{\alpha,\beta,\iota,\jmath}^\gamma$ be the set of $n$-tuples $(S_1, \dots, S_n)$ where $S_i$ is a skew edge-labeled tableaux of shape $\gamma_i/\alpha_{\iota^{-1}(i)}$ and content $\beta_{\jmath^{-1}(i)}$. We take the convention that $\alpha_{\iota^{-1}(i)}=0$ if $i \notin \im \iota$, and similarly for $\beta_{\jmath^{-1}(i)}$.
Note that if $[n] \neq \im \iota \cup \im \jmath$, then the set $\mathcal{S}_{\alpha,\beta,\iota,\jmath}^\gamma$ is empty.
For $(S_1, \dots, S_n) \in \mathcal{S}_{\alpha,\beta,\iota,\jmath}^\gamma$,
we let
\[
\wt(S_1,\dots,S_n) \coloneqq \prod_{i=1}^n\wt(S_i).
\]
Given an $n$-tuple $(S_1, \dots, S_n) \in \mathcal{S}_{\alpha,\beta,\iota,\jmath}^\gamma$, it is convenient to stack the $S_i$ on each other with $S_1$ at the top. We call such a stack a \newword{skyline edge-labeled tableau}.
\end{definition}

\begin{example}\label{ex:skyline}
	Let $\alpha = (3,2)$, $\beta = (2,3)$, and $\gamma = (3,2,4)$. Set $\iota(1) = 1$ and $\iota(2) = 3$. Also set $\jmath(1) = 2$ and $\jmath(2) = 3$. The unique skew edge-labeled tableau of shape $3/3$ and content $0$ is $\Scale[0.7]{\ytableaushort{\blank \blank \blank}}$. The unique skew edge-labeled tableau of shape $2/0$ and content $2$ is $\Scale[0.7]{\ytableaushort{\bullet \bullet}}$. There are exactly two skew edge-labeled tableaux of shape $4/2$ and content $3$:
	\Scale[0.7]{\begin{picture}(100,33)
\ytableausetup{boxsize=2.0em}
\put(8,0){$\ytableaushort{ \blank \blank \bullet \bullet}$}
\put(16,-3){$\bullet$}
\end{picture}}
and
\Scale[0.7]{\begin{picture}(92,33)
\ytableausetup{boxsize=2.0em}
\put(8,0){$\ytableaushort{ \blank \blank \bullet \bullet}$}
\put(37,-3){$\bullet$}
\end{picture}}. Thus, we have
\[
\mathcal{S}_{\alpha,\beta,\iota,\jmath}^\gamma = \left\{
\Scale[0.8]{\begin{picture}(100,33)
\ytableausetup{boxsize=2.0em}
\put(8,15){$\ytableaushort{ \blank \blank \blank,\bullet \bullet,\blank \blank \bullet \bullet}$}
\put(16,-28.5){$\bullet$}
\end{picture}}
,
\Scale[0.8]{\begin{picture}(100,33)
\ytableausetup{boxsize=2.0em}
\put(8,15){$\ytableaushort{ \blank \blank \blank,\bullet \bullet,\blank \blank \bullet \bullet}$}
\put(37,-28.5){$\bullet$}
\end{picture}}
\right\}.
\]
The weight of the first skyline edge-labeled tableau is $y_1 - y_4$, while the weight of the second is $y_2 - y_5$.
\end{example}

We then have the following combinatorial rule for the multiplication of double monomial quasisymmetric functions.

\begin{theorem}\label{thm:product-formula}
Let $\alpha=(\alpha_1, \dots, \alpha_\ell)$ and $\beta = (\beta_1, \dots, \beta_m)$ be compositions. Then
\[
M_\alpha(\bx,\by)\cdot M_\beta(\bx,\by)=\sum_\gamma c^\gamma_{\alpha,\beta} M_\gamma(\bx,\by),
\]
where the sum runs over compositions $\gamma = (\gamma_1, \dots, \gamma_n)$ and the structure coefficient $c^\gamma_{\alpha,\beta}$ is given by
\[
c^\gamma_{\alpha,\beta}=\sum_{\iota,\jmath}\sum_{(S_1,\dots,S_n)\in\mathcal{S}_{\alpha,\beta,\iota,\jmath}^\gamma} \!\!\!\!\!\!\!\!\!\!\wt(S_1,\dots,S_n),
\]
where $\iota : [\ell] \to [n]$ and $\jmath : [m] \to [n]$ range over order-preserving injections.
\end{theorem}
\begin{proof}
We have
\[
M_\alpha(\bx,\by)=\sum_{i_1<\dots<i_\ell}\prod_{L=1}^\ell\prod_{i'_L\leq\alpha_L}(x_{i_L}-y_{i'_L})
\]
and
\[
M_\beta(\bx,\by)=\sum_{j_1<\dots<j_m}\prod_{M=1}^m\prod_{j'_M\leq\beta_M}(x_{j_M}-y_{j'_M}).
\]
Given any choices $i_1<\dots<i_\ell$ and $j_1<\dots<j_m$, we uniquely obtain $k_1<\dots<k_n$ with
\[
\{i_1,\dots,i_\ell,j_1,\dots,j_m\}=\{k_1,\dots,k_n\}.
\]
 We also obtain order-preserving injections $\iota\colon[\ell]\to[n]$ and $\jmath\colon[m]\to [n]$ defined by $i_L=k_{\iota(L)}$ and $j_M=k_{\jmath(M)}$. We see then that $M_\alpha(\bx,\by)\cdot M_\beta(\bx,\by)$ is the sum over all choices of order preserving injections $\iota,\jmath$ of
\begin{equation}\label{eqn:product-formula1}
\sum_{k_1<\dots<k_n}\prod_{L=1}^\ell\prod_{M=1}^m \prod_{i'_L\leq\alpha_L}(x_{k_{\iota(L)}}-y_{i'_L})\cdot \prod_{j'_M\leq\beta_M}(x_{k_{\jmath(M)}}-y_{j'_M}).
\end{equation}
This is also equal to
\begin{equation}\label{eqn:product-formula2}
\sum_{k_1<\dots<k_n}\prod_{N=1}^n \prod_{{i'_N}\leq\alpha_{\iota^{-1}(N)}}(x_{k_{N}}-y_{i'_N})\cdot \prod_{j'_N\leq\beta_{\jmath^{-1}(N)}}(x_{k_{N}}-y_{j'_N}),
\end{equation}
where recall our convention that if $N$ is not in the image of $\iota$, then $\alpha_{\iota^{-1}(N)}=0$ and hence the product $\prod_{{i'_N}\leq\alpha_{\iota^{-1}(N)}}(x_{k_{N}}-y_{i'_N})$ is $1$; similarly if $N$ is not in the image of $\jmath$.

In order to write this expression as a sum of double monomial quasisymmetric functions, we must understand how to multiply $\prod_{{i'_N}\leq\alpha_{\iota^{-1}(N)}}(x_{k_{N}}-y_{i'_N})$ and $\prod_{j'_N\leq\beta_{\jmath^{-1}(N)}}(x_{k_{N}}-y_{j'_N})$ when $N$ is in the image of both $\iota$ and $\jmath$. For this, recall that $\prod_{i'\leq\alpha_{\iota^{-1}(N)}}(x-y_{i'})$ is nothing other than the equivariant cohomology class associated to the cell $e_{2\alpha_{\iota^{-1}(N)}}$ of projective space, and similarly $\prod_{j'\leq\beta_{\jmath^{-1}(N)}}(x-y_{j'})$ is the equivariant cohomology class associated to the cell $e_{2\beta_{\jmath^{-1}(N)}}$. Thus, by the equivariant Littlewood--Richardson rule of Theorem~\ref{thm:Thomas-Yong}, equation \eqref{eqn:product-formula2} simplifies as
\[
\sum_{k_1<\dots<k_n}\sum_{\gamma_1,\dots,\gamma_n}\sum_{S_1,\dots,S_n}\prod_{N=1}^n\wt(S_N) \prod_{{k'_N}\leq \gamma_N}(x_{k_{N}}-y_{k'_N}),
\]
where $c_1,\dots,c_n$ are positive integers and $S_N$ is a skew edge-labeled tableau of shape $\gamma_N/\alpha_{\iota^{-1}(N)}$ and content $\beta_{\jmath^{-1}(N)}$. Thus, letting $\gamma=(\gamma_1,\dots,\gamma_n)$ we have arrived at our desired formula.
\end{proof}

We note that an alternative rule to Theorem~\ref{thm:product-formula} can be obtained by replacing use of \cite{Thomas.Yong18} with any other equivariant Littlewood--Richardson rule, for example those of \cite{Kreiman} or \cite{Knutson.Tao:HT}.

\begin{example}\label{ex:LR}
	Let $\alpha = (3,2)$, $\beta = (2,3)$, and $\gamma = (3,2,4)$, as in Example~\ref{ex:skyline}. We compute the equivariant structure coefficient $c^\gamma_{\alpha,\beta}$ by Theorem~\ref{thm:product-formula}. A priori, there are three injections $\iota$ to consider. However, since $\alpha_1 > \gamma_2$, there are no skyline edge-labeled tableaux if $\iota(1) = 2$. Hence, it suffices to consider the injection $\iota_1 : [2] \to [3]$ defined by $\iota_1(1) = 1$ and $\iota_1(2) = 2$ and the injection $\iota_2$ defined by $\iota_2(1) = 1$ and $\iota_2(2) = 3$.  Similarly, there are no skyline edge-labeled tableaux if $\jmath(2) = 2$. Hence, it suffices to consider the injection $\jmath_1$ defined by $\jmath_1(1) = 1$ and $\jmath_1(2) = 3$ and the injection $\jmath_2$ defined by $\jmath_2(1) = 2$ and $\jmath_2(2) = 3$. Note that $\iota_2$ and $\jmath_2$ are the maps $\iota$ and $\jmath$ of Example~\ref{ex:skyline}, respectively.

		For the pair of injections $\iota_1, \jmath_1$, there are no skyline edge-labeled tableaux, as there is no skew edge-labeled tableau of shape $4/0$ and content only $3$. For the pair $\iota_1, \jmath_2$, there are also no skyline edge-labeled tableaux, for the same reason.
	For the pair of injections $\iota_2, \jmath_1$, there are no skyline edge-labeled tableaux, as $2$ is not in the image of either injection. Finally, for the pair $\iota_2, \jmath_2$, we have the skyline edge-labeled tableaux computed in Example~\ref{ex:skyline}, contributing a total weight of $(y_1-y_4) + (y_2 - y_5)$.

	Thus, $c^{(3,2,4)}_{(3,2),(2,3)} = (y_1-y_4) + (y_2 - y_5)$.
\end{example}

\bibliographystyle{amsalpha}
\bibliography{quasi-symm}

@book {AndersonFulton,
    AUTHOR = {Anderson, David and Fulton, William},
     TITLE = {Equivariant cohomology in algebraic geometry},
    SERIES = {Cambridge Studies in Advanced Mathematics},
    VOLUME = {210},
 PUBLISHER = {Cambridge University Press, Cambridge},
      YEAR = {2024},
     PAGES = {xv+446},
      ISBN = {978-1-00-934998-7},
   MRCLASS = {14L30 (05E14 14-02 14Mxx 20G05 55N91)},
  MRNUMBER = {4655919},
}

@article {Baker.Richter,
    AUTHOR = {Baker, Andrew and Richter, Birgit},
     TITLE = {Quasisymmetric functions from a topological point of view},
   JOURNAL = {Math. Scand.},
  FJOURNAL = {Mathematica Scandinavica},
    VOLUME = {103},
      YEAR = {2008},
    NUMBER = {2},
     PAGES = {208--242},
}

@incollection {Billey.McNamara,
    AUTHOR = {Billey, Sara C. and McNamara, Peter R. W.},
     TITLE = {The contributions of {S}tanley to the fabric of symmetric and
              quasisymmetric functions},
 BOOKTITLE = {The mathematical legacy of {R}ichard {P}. {S}tanley},
     PAGES = {83--104},
 PUBLISHER = {Amer. Math. Soc., Providence, RI},
      YEAR = {2016},
}

@article {Brion97,
    AUTHOR = {Brion, M.},
     TITLE = {Equivariant {C}how groups for torus actions},
   JOURNAL = {Transform. Groups},
  FJOURNAL = {Transformation Groups},
    VOLUME = {2},
      YEAR = {1997},
    NUMBER = {3},
     PAGES = {225--267},
      ISSN = {1083-4362,1531-586X},
   MRCLASS = {14C15 (14L30 14M15)},
  MRNUMBER = {1466694},
MRREVIEWER = {Dan\ Edidin},
       DOI = {10.1007/BF01234659},
       URL = {https://doi.org/10.1007/BF01234659},
}

@article {Buch,
    AUTHOR = {Buch, Anders Skovsted},
     TITLE = {A {L}ittlewood-{R}ichardson rule for the {$K$}-theory of
              {G}rassmannians},
   JOURNAL = {Acta Math.},
  FJOURNAL = {Acta Mathematica},
    VOLUME = {189},
      YEAR = {2002},
    NUMBER = {1},
     PAGES = {37--78},
}

@article {ChangSkjelbred,
    AUTHOR = {Chang, Theodore and Skjelbred, Tor},
     TITLE = {Topological {S}chur lemma and related results},
   JOURNAL = {Bull. Amer. Math. Soc.},
  FJOURNAL = {Bulletin of the American Mathematical Society},
    VOLUME = {79},
      YEAR = {1973},
     PAGES = {1036--1038},
      ISSN = {0002-9904},
   MRCLASS = {57E15 (55F20)},
  MRNUMBER = {334263},
MRREVIEWER = {J.\ A.\ Schafer},
       DOI = {10.1090/S0002-9904-1973-13320-8},
       URL = {https://doi.org/10.1090/S0002-9904-1973-13320-8},
}

@article {Choi.Kim.Nam.Oh,
    AUTHOR = {Choi, Seung-Il and Kim, Young-Hun and Nam, Sun-Young and Oh,
              Young-Tak},
     TITLE = {Homological properties of 0-{H}ecke modules for dual
              immaculate quasisymmetric functions},
   JOURNAL = {Forum Math. Sigma},
  FJOURNAL = {Forum of Mathematics. Sigma},
    VOLUME = {10},
      YEAR = {2022},
     PAGES = {Paper No. e91, 37 pages},
}

@article {Corteel.Haglund.Mandelshtam.Mason.Williams,
    AUTHOR = {Corteel, Sylvie and Haglund, Jim and Mandelshtam, Olya and
              Mason, Sarah and Williams, Lauren},
     TITLE = {Compact formulas for {M}acdonald polynomials and
              quasisymmetric {M}acdonald polynomials},
   JOURNAL = {Selecta Math. (N.S.)},
  FJOURNAL = {Selecta Mathematica. New Series},
    VOLUME = {28},
      YEAR = {2022},
    NUMBER = {2},
     PAGES = {Paper No. 32, 33 pages},
}

@article {Diehl.EbrahimiFard.Tapia,
    AUTHOR = {Diehl, Joscha and Ebrahimi{-}Fard, Kurusch and Tapia, Nikolas},
     TITLE = {Tropical time series, iterated-sums signatures, and
              quasisymmetric functions},
   JOURNAL = {SIAM J. Appl. Algebra Geom.},
  FJOURNAL = {SIAM Journal on Applied Algebra and Geometry},
    VOLUME = {6},
      YEAR = {2022},
    NUMBER = {4},
     PAGES = {563--599},
}

@article {Duchamp.Krob.Leclerc.Thibon,
    AUTHOR = {Duchamp, G\'{e}rard and Krob, Daniel and Leclerc, Bernard and
              Thibon, Jean-Yves},
     TITLE = {Fonctions quasi-sym\'{e}triques, fonctions sym\'{e}triques non
              commutatives et alg\`ebres de {H}ecke \`a {$q=0$}},
   JOURNAL = {C. R. Acad. Sci. Paris S\'{e}r. I Math.},
  FJOURNAL = {Comptes Rendus de l'Acad\'{e}mie des Sciences. S\'{e}rie I.
              Math\'{e}matique},
    VOLUME = {322},
      YEAR = {1996},
    NUMBER = {2},
     PAGES = {107--112},
}

@article {Edidin.Graham,
    AUTHOR = {Edidin, Dan and Graham, William},
     TITLE = {Equivariant intersection theory},
   JOURNAL = {Invent. Math.},
  FJOURNAL = {Inventiones Mathematicae},
    VOLUME = {131},
      YEAR = {1998},
    NUMBER = {3},
     PAGES = {595--634},
      ISSN = {0020-9910,1432-1297},
   MRCLASS = {14C17 (14F99)},
  MRNUMBER = {1614555},
MRREVIEWER = {Burt\ Totaro},
       DOI = {10.1007/s002220050214},
       URL = {https://doi.org/10.1007/s002220050214},
}

@article{FranzCS,
author = {Franz, Matthias},
title = {The {C}hang--{S}kjelbred {L}emma and Generalizations},
journal = {preprint},
note =     {\arxiv{2306.05165}},
year = {2023},
pages = {9 pages}
}

@article {GabberLiuLorenzini,
    AUTHOR = {Gabber, Ofer and Liu, Qing and Lorenzini, Dino},
     TITLE = {The index of an algebraic variety},
   JOURNAL = {Invent. Math.},
  FJOURNAL = {Inventiones Mathematicae},
    VOLUME = {192},
      YEAR = {2013},
    NUMBER = {3},
     PAGES = {567--626},
      ISSN = {0020-9910,1432-1297},
   MRCLASS = {14G05 (13H15)},
  MRNUMBER = {3049930},
MRREVIEWER = {Damian\ R\"{o}ssler},
       DOI = {10.1007/s00222-012-0418-z},
       URL = {https://doi.org/10.1007/s00222-012-0418-z},
}

@incollection {Gessel,
    AUTHOR = {Gessel, Ira M.},
     TITLE = {Multipartite {$P$}-partitions and inner products of skew
              {S}chur functions},
 BOOKTITLE = {Combinatorics and algebra ({B}oulder, {C}olo., 1983)},
    SERIES = {Contemp. Math.},
    VOLUME = {34},
     PAGES = {289--317},
 PUBLISHER = {Amer. Math. Soc., Providence, RI},
      YEAR = {1984},
}

@article {GKM,
    AUTHOR = {Goresky, Mark and Kottwitz, Robert and MacPherson, Robert},
     TITLE = {Equivariant cohomology, {K}oszul duality, and the localization
              theorem},
   JOURNAL = {Invent. Math.},
  FJOURNAL = {Inventiones Mathematicae},
    VOLUME = {131},
      YEAR = {1998},
    NUMBER = {1},
     PAGES = {25--83},
}

@article {Graham,
    AUTHOR = {Graham, William},
     TITLE = {Positivity in equivariant {S}chubert calculus},
   JOURNAL = {Duke Math. J.},
  FJOURNAL = {Duke Mathematical Journal},
    VOLUME = {109},
      YEAR = {2001},
    NUMBER = {3},
     PAGES = {599--614},
}

@book {HartshorneAG,
    AUTHOR = {Hartshorne, Robin},
     TITLE = {Algebraic geometry},
    SERIES = {Graduate Texts in Mathematics},
    VOLUME = {52},
 PUBLISHER = {Springer-Verlag, New York-Heidelberg},
      YEAR = {1977},
     PAGES = {xvi+496},
      ISBN = {0-387-90244-9},
   MRCLASS = {14-01},
  MRNUMBER = {463157},
MRREVIEWER = {Robert\ Speiser},
}

@article {Hazewinkel,
    AUTHOR = {Hazewinkel, Michiel},
     TITLE = {The algebra of quasi-symmetric functions is free over the
              integers},
   JOURNAL = {Adv. Math.},
  FJOURNAL = {Advances in Mathematics},
    VOLUME = {164},
      YEAR = {2001},
    NUMBER = {2},
     PAGES = {283--300},
}

@article {James,
    AUTHOR = {James, I. M.},
     TITLE = {Reduced product spaces},
   JOURNAL = {Ann. of Math. (2)},
  FJOURNAL = {Annals of Mathematics. Second Series},
    VOLUME = {62},
      YEAR = {1955},
     PAGES = {170--197},
}

@article {Knutson.Tao:HT,
    AUTHOR = {Knutson, Allen and Tao, Terence},
     TITLE = {Puzzles and (equivariant) cohomology of {G}rassmannians},
   JOURNAL = {Duke Math. J.},
  FJOURNAL = {Duke Mathematical Journal},
    VOLUME = {119},
      YEAR = {2003},
    NUMBER = {2},
     PAGES = {221--260},
}

@article {Kreiman,
    AUTHOR = {Kreiman, Victor},
     TITLE = {Equivariant {L}ittlewood-{R}ichardson skew tableaux},
   JOURNAL = {Trans. Amer. Math. Soc.},
  FJOURNAL = {Transactions of the American Mathematical Society},
    VOLUME = {362},
      YEAR = {2010},
    NUMBER = {5},
     PAGES = {2589--2617},
}

@article {Lam.Pylyavskyy,
    AUTHOR = {Lam, Thomas and Pylyavskyy, Pavlo},
     TITLE = {Combinatorial {H}opf algebras and {$K$}-homology of
              {G}rassmannians},
   JOURNAL = {Int. Math. Res. Not. IMRN},
  FJOURNAL = {International Mathematics Research Notices. IMRN},
      YEAR = {2007},
    NUMBER = {24},
     PAGES = {Art. ID rnm125, 48 pages},
}

@incollection {Mason,
    AUTHOR = {Mason, Sarah K.},
     TITLE = {Recent trends in quasisymmetric functions},
 BOOKTITLE = {Recent trends in algebraic combinatorics},
    SERIES = {Assoc. Women Math. Ser.},
    VOLUME = {16},
     PAGES = {239--279},
 PUBLISHER = {Springer, Cham},
      YEAR = {2019},
}

@article{Pechenik.SatrianoKthy,
author = {Pechenik, Oliver and Satriano, Matthew},
title = {Combinatorial models for the cohomology and {$K$}-theory of some loop spaces},
journal = {preprint},
note =     {\arxiv{2205.12415}},
year = {2022},
pages = {30 pages}
}

@incollection {Pechenik.Searles:survey,
    AUTHOR = {Pechenik, Oliver and Searles, Dominic},
     TITLE = {Asymmetric function theory},
 BOOKTITLE = {Schubert calculus and its applications in combinatorics and
              representation theory},
    SERIES = {Springer Proc. Math. Stat.},
    VOLUME = {332},
     PAGES = {73--112},
 PUBLISHER = {Springer, Singapore},
      YEAR = {2020},
}

@article {Pechenik.Yong,
    AUTHOR = {Pechenik, Oliver and Yong, Alexander},
     TITLE = {Equivariant {$K$}-theory of {G}rassmannians},
   JOURNAL = {Forum Math. Pi},
  FJOURNAL = {Forum of Mathematics. Pi},
    VOLUME = {5},
      YEAR = {2017},
     PAGES = {e3, 128 pages},
}

@article{Rydh.dividedpowers,
author = {Rydh, David},
title = {Families of zero cycles and divided powers:~{I}.~{R}epresentability},
journal = {preprint},
note =     {\arxiv{0803.0618}},
pages ={52 pages},
year = {2008}
}

@article {Searles:extended,
    AUTHOR = {Searles, Dominic},
     TITLE = {Indecomposable {$0$}-{H}ecke modules for extended {S}chur
              functions},
   JOURNAL = {Proc. Amer. Math. Soc.},
  FJOURNAL = {Proceedings of the American Mathematical Society},
    VOLUME = {148},
      YEAR = {2020},
    NUMBER = {5},
     PAGES = {1933--1943},
}

@article {Shareshian.Wachs,
    AUTHOR = {Shareshian, John and Wachs, Michelle L.},
     TITLE = {Chromatic quasisymmetric functions},
   JOURNAL = {Adv. Math.},
  FJOURNAL = {Advances in Mathematics},
    VOLUME = {295},
      YEAR = {2016},
     PAGES = {497--551},
}

@book {Stanley,
    AUTHOR = {Stanley, Richard P.},
     TITLE = {Ordered structures and partitions},
    SERIES = {Mem. Amer. Math. Soc., No. 119},
 PUBLISHER = {American Mathematical Society, Providence, R.I.},
      YEAR = {1972},
     PAGES = {iii+104},
}

@article {Tewari.vanWilligenburg,
    AUTHOR = {Tewari, Vasu V. and van Willigenburg, Stephanie J.},
     TITLE = {Modules of the 0-{H}ecke algebra and quasisymmetric {S}chur
              functions},
   JOURNAL = {Adv. Math.},
  FJOURNAL = {Advances in Mathematics},
    VOLUME = {285},
      YEAR = {2015},
     PAGES = {1025--1065},
}

@article {Thomas.Yong18,
    AUTHOR = {Thomas, Hugh and Yong, Alexander},
     TITLE = {Equivariant {S}chubert calculus and jeu de taquin},
   JOURNAL = {Ann. Inst. Fourier (Grenoble)},
  FJOURNAL = {Universit\'{e} de Grenoble. Annales de l'Institut Fourier},
    VOLUME = {68},
      YEAR = {2018},
    NUMBER = {1},
     PAGES = {275--318},
}

@article {Vakil06,
    AUTHOR = {Vakil, Ravi},
     TITLE = {A geometric {L}ittlewood-{R}ichardson rule},
      NOTE = {Appendix A written with A. Knutson},
   JOURNAL = {Ann. of Math. (2)},
  FJOURNAL = {Annals of Mathematics. Second Series},
    VOLUME = {164},
      YEAR = {2006},
    NUMBER = {2},
     PAGES = {371--421},
      ISSN = {0003-486X,1939-8980},
   MRCLASS = {05E10 (14N15)},
  MRNUMBER = {2247964},
MRREVIEWER = {Grant\ Walker},
       DOI = {10.4007/annals.2006.164.371},
       URL = {https://doi.org/10.4007/annals.2006.164.371},
}

@article {VoevodskyHomology,
    AUTHOR = {Voevodsky, V.},
     TITLE = {Homology of schemes},
   JOURNAL = {Selecta Math. (N.S.)},
  FJOURNAL = {Selecta Mathematica. New Series},
    VOLUME = {2},
      YEAR = {1996},
    NUMBER = {1},
     PAGES = {111--153},
      ISSN = {1022-1824,1420-9020},
   MRCLASS = {14F99},
  MRNUMBER = {1403354},
MRREVIEWER = {Claudio\ Pedrini},
       DOI = {10.1007/BF01587941},
       URL = {https://doi.org/10.1007/BF01587941},
}

@book {Whitehead,
    AUTHOR = {Whitehead, George W.},
     TITLE = {Elements of homotopy theory},
    SERIES = {Graduate Texts in Mathematics},
    VOLUME = {61},
 PUBLISHER = {Springer-Verlag, New York-Berlin},
      YEAR = {1978},
     PAGES = {xxi+744},
      ISBN = {0-387-90336-4},
}

\end{document}